\documentclass[a4paper,12pt]{amsart}
\usepackage{amsfonts,amsmath,amssymb,amsthm}
\usepackage{latexsym}
\usepackage{eucal}
\usepackage[dvips]{graphics}
\usepackage{mathabx}
\usepackage[english]{babel}
\usepackage{epsfig}
\usepackage{enumerate}
\usepackage[colorlinks=true, linkcolor=red, citecolor=blue]{hyperref}
\newcommand{\nwc}{\newcommand}
\nwc\eps{\varepsilon}

\newcommand\I{{\mathcal{I}}}

\newcommand\R{\mathbb{R}}

\nwc\EN{{\mathcal{E}_1}}

\newcommand\ep{\epsilon}

\nwc{\dx}{\partial_x}
\nwc{\dy}{\partial_y}

\newcommand\J{{\mathcal{J}}}

\nwc{\hamone}{{\mathcal{H}}}
\nwc{\Imu}{\Sigma}
\nwc{\Iem}{\I}
\nwc{\Gep}{G}
\nwc{\wt}{\widetilde}
\nwc{\IF}{\mathcal F}

\newtheorem{thm}{Theorem}[section]
 
 \newtheorem{lem}[thm]{Lemma}

 \newtheorem{defn}[thm]{Definition}
 \newtheorem{claim}{Claim}
 \newtheorem{rem}[thm]{Remark}

 \numberwithin{equation}{section}
\makeatletter
\@namedef{subjclassname@2010}{\textup{2020} Mathematics Subject Classification}
\makeatother

\begin{document}

\title[stability for generalized $abcd$-Boussinesq system]
 {Stability of solitary waves for generalized $abcd$-Boussinesq system: The Hamiltonian case}

\author[Capistrano-Filho]{Roberto de A.  Capistrano--Filho*}
\address{
Departamento de Matem\'atica, Universidade Federal de Pernambuco\\
Cidade Universit\'aria, 50740-545, Recife (PE), Brazil}
\email{roberto.capistranofilho@ufpe.br}



\author[Quintero]{José Raul Quintero}
\address{
Department of Mathematics, Universidad del Valle\\
Calle 13, 100-00, Cali, Colombia}
\email{jose.quintero@correounivalle.edu.co}



\author[Sun]{Shu-Ming Sun}
\address{
Department of Mathematics, Virginia Tech\\
Blacksburg (VA) 24061, USA}
\email{sun@math.vt.edu}

\date{\today}
\thanks{*Corresponding author: roberto.capistranofilho@ufpe.br}
\thanks{Capistrano-Filho was supported by CNPq grants numbers 307808/2021-1, 401003/2022-1, and 200386/2022-0, CAPES  grant number 88887.520204/2020-00, MathAmSud 21-Math-03, and CAPES/COFECUB grant number 88887.879175/2023-00. Quintero was supported by Universidad del Valle (project CI 71231) and by Minciencias-Colombia under the project MathAmSud 21-Math-03. Sun was supported by a grant from the Simons Foundation (712822, SMS)}




\numberwithin{equation}{section}

\begin{abstract}
The $abcd$-Boussinesq system is a model of two equations that can describe the propagation of small-amplitude long waves in both directions in the water of finite depth. Considering the Hamiltonian regimes, where the parameters $b$ and $d$ in the system satisfy $b=d>0$, small solutions in the energy space are globally defined. Then, a variational approach is applied to establish the existence and nonlinear stability of the set of solitary-wave solutions for the generalized $abcb$-Boussinesq system. The main point of the analysis is to show that the traveling-wave solutions of the generalized $abcb$-Boussinesq system converge to nontrivial solitary-wave solutions of the generalized Korteweg-de Vries equation. Moreover, if $p$ is the exponent of the nonlinear terms for the generalized $abcb$-Boussinesq system, then the nonlinear stability of the set of solitary-waves is obtained for any $p$ with  $ 1 \leq p < p_0$ where $p_0 $ is strictly larger than $4$, while it has been known that the critical exponent for the stability of solitary waves of the generalized KdV equations is equal to $ 4$.

\end{abstract}

\maketitle

\allowdisplaybreaks

\section{Introduction}
\subsection{$abcd$-Boussinesq model}
Boussinesq \cite{boussinesq1} introduced several nonlinear partial differential equations to explain certain physical observations concerning the water waves, where the surface tension has been neglected, e.g., the emergence and stability of solitary waves. Unfortunately, several systems derived by Boussinesq were shown to be ill-posed, and thus, there was a need to propose other systems with better mathematical properties.  In that direction, the four-parameter family of the Boussinesq system
\begin{equation}
\left\{
\begin{array}
[c]{l}%
\eta_{t}+\partial_xu+\partial_x(   \eta u) +a\partial_{xxx}u-b\partial_{xx}\eta_{t}=0\text{,}\\
u_{t}+\partial_x\eta+u\partial_xu+c\partial_{xxx}\eta-d\partial_{xx}u_{t}=0,
\end{array}
\right.  \label{int_29e}%
\end{equation}
was introduced by Bona, Chen, and Saut \cite{BCS02}  to describe the motion of small-amplitude long waves on the surface of an ideal fluid of finite depth under gravity and in situations where the motion is sensibly two-dimensional. In \eqref{int_29e}, $\eta$ is the elevation of the fluid surface from the equilibrium position, and $u$ is the horizontal velocity at a certain height in the flow. Initially, the constants $a, b, c, d$ must satisfy only the following relation
$$
a+b+c+d=\frac{1}{3}-\sigma
$$
where $\sigma \geq 0$ is the surface tension coefficient of the fluid. As reported in \cite{BCS02}, when $\sigma$ is zero, parameters $a, b, c, d$ must satisfy the relations
\begin{equation}\label{abcd}
a+b=\frac{1}{2}\left(\theta^2-\frac{1}{3}\right)\quad  c+d=\frac{1}{2}\left(1-\theta^2\right) \geq 0, \quad
a+b+c+d=\frac{1}{3},
\end{equation}
where $\theta \in[0,1]$. In addition, $a, b, c, d$ can be rewritten in the form
\begin{equation}\label{abcd_1}
\begin{split}
a & =\frac{1}{2}\left(\theta^2-\frac{1}{3}\right) \nu, \quad b  =\frac{1}{2}\left(\theta^2-\frac{1}{3}\right)(1-\nu), \\
c & =\frac{1}{2}\left(1-\theta^2\right) \mu, \quad d   =\frac{1}{2}\left(1-\theta^2\right)(1-\mu),
\end{split}
\end{equation}
with $\nu, \mu$ are suitable real parameters in the sense that \eqref{abcd_1} implies \eqref{abcd}. Depending on the choice of different real values for $\nu, \mu$ and $\theta \in[0,1]$, it is possible to deduce some classical systems, such as the classical Boussinesq system, Kaup system, Bona-Smith system, coupled Benjamin-Bona-Mahony system, coupled Korteweg-de Vries system, and coupled mixed Korteweg-de Vries-Benjamin-Bona-Mahony systems.

The authors in \cite{BCS04} studied the initial value problem for the system \eqref{int_29e}. The well-posedness on $\mathbb R$ was shown if and only if the parameters $a,b,c,d$ are in the following regimes
\begin{eqnarray*}
(\textrm{C}1)&& \ b=d>0,\ a\le 0,\ c< 0;\\
(\textrm{C}2)&& \ b,d\ge 0, \ a=c>0.
\end{eqnarray*}
Thus, observe that in the (C1) case, the system \eqref{int_29e} takes the form as
\begin{equation} \left\{
\begin{array}{ll}
\left(I-b\partial_x^2\right) \eta_t
+\partial_x u+a\partial_x^3 u
+\partial_x\left(\eta u \right)=0,& (x,t)\in\mathbb{R}\times\mathbb{R},
\\ \label{1bbl1}
\left(I- b\partial^2_x\right) u_t  + \eta_x
+c\partial_x^3\eta
+u\partial_x u=0,& (x,t)\in\mathbb{R}\times\mathbb{R},\\
\eta(x,0)=\eta_0(x), \quad u(x,0)=u_0(x), &x\in\mathbb{R}.
\end{array}\right.
\end{equation}
It is known that system \eqref{1bbl1} admits (big) solitary-wave solutions in certain regimes of the parameters involved in the system (for instance, see \cite{BaoChenLiu2015} and references therein for details). Moreover, when $b=d>0$, it was also shown in \cite{BCS04} that the system \eqref{1bbl1} is Hamiltonian and globally well-posed in the energy space $X= H^1(\mathbb{R}) \times H^1(\mathbb{R})$, at least for small data, thanks to the conservation of the energy
$$
\begin{aligned}
\mathcal{H}[\eta,u](t):= & \frac{1}{2} \int\left(-a\left(\partial_x u\right)^2-c\left(\partial_x \eta\right)^2+u^2+\eta^2+u^2 \eta\right)(t, x) d x.
\end{aligned}
$$

\subsection{Problem setting} Keeping the previous conservation law in mind, our goal is to investigate the existence and stability of some traveling-wave solutions for a more general nonlinear dispersive system associated with \eqref{1bbl1}, namely
\begin{equation} \left\{
\begin{array}{ll}
\left(I-b\partial_x^2\right) \eta_t
+\partial_x u+a\partial_x^3 u
+\partial_x\left(\eta u^p \right)=0,& (x,t)\in\mathbb{R}\times\mathbb{R},
\\ \label{1bbl}
\left(I- b\partial^2_x\right) u_t  + \partial _x \eta
+c\partial_x^3\eta
+\frac{1}{p+1}\partial_x (u^{p+1})=0,& (x,t)\in\mathbb{R}\times\mathbb{R},
\\
\eta(x,0)=\eta_0(x), \quad u(x,0)=u_0(x), &x\in\mathbb{R}.
\end{array}\right.
\end{equation}
Here, $\eta=\eta(x,t)$ and $u=u(x,t)$ are real-valued
functions,  $p> 0  $ is a rational constant of the form
\begin{align}
p = \frac{p_1}{p_2} \quad \mbox{with}\quad (p_1,\, p_2) = 1\quad \mbox{and}\quad  p_1, \, p_2 \ \mbox{ odd}, \label{assump_p}
\end{align}
and the parameters $a, b, c, d$ satisfy (C1). In the following, the system \eqref{1bbl} is called the generalized $abcb$-Boussinesq system since $b=d$ in \eqref{int_29e}\footnote{A robust analysis for the well-posedness problem of \eqref{1bbl} was made in \cite{BCS04} for the case $p=1$.}.

It is well understood that the general stability theory developed in \cite{GSS} is a powerful tool to prove the stability of solitary-wave solutions for abstract Hamiltonian systems. Taking it into account, roughly speaking, we are interested in the study of the following problem:

\vspace{0.2cm}

\noindent\textbf{Orbital Stability Problem:}\textit{ Let $\omega \in \mathbb{R}^+$ and $\varepsilon>0$ be given and $(\tilde \eta_\omega, \tilde u_\omega )$ be a traveling-wave solution of \eqref{1bbl} with traveling speed $\omega$. Is there $\delta(\varepsilon)>0$ such that for $(\eta_0,u_0)\in H^1(\mathbb{R})\times H^1(\mathbb{R})$ with
$$
\|(\eta_0,u_0)-(\tilde \eta_\omega,\tilde u_\omega)\|_{ X}<\delta(\varepsilon),
$$
there exists a unique global solution $(\eta(\cdot,t),u(\cdot,t))$ of the system \eqref{1bbl} such that
\begin{equation}\label{OSP}
\inf_{y\in\mathbb{R}}\|(\eta(\cdot,t),u(\cdot,t))-(\tilde \eta_\omega(\cdot+y),\tilde u_\omega(\cdot+y))\|_{X}<\varepsilon  \ \
\text{for all} \ \ t> 0 ?
\end{equation}
Here, we may let a set $ \tilde {\mathcal G}_\omega  = \{ (\tilde \eta_\omega (\cdot + y ) , \tilde u_\omega (\cdot + y ))\, |\, y \in \mathbb{R} \}$. Then, the orbital stability can be stated as the set stability: \eqref{OSP} is equivalent to $\mbox{dist} ( (\eta(\cdot,t),u(\cdot,t)), \tilde {\mathcal G}_\omega ) < \varepsilon$ for all $t > 0 $.}

\vspace{0.2cm}

To solve the previous problem, it is natural to use the existence of a Hamiltonian structure, as mentioned before. Thus, for our analysis of the stability, we consider the Hamiltonian structure\footnote{The Hamiltonian structure comes from the fact that $\mathcal{J}_{bb}$ defined in the following is skew-symmetric as pointed out in \cite[Section 4]{BCS04}.} for the generalized $abcb$-Boussinesq system \eqref{1bbl} given by
\begin{equation}\label{ham}
\mathcal{H}\begin{pmatrix}\eta \\ u
\end{pmatrix}=\frac{1}{2}\int_{\mathbb{R}}
\left(\eta^2- c (\partial_x \eta)^2+u^{2}- a (\partial_x u)^2+\frac{2}{p+1}\eta u^{p+1}\right)dx.
\end{equation}
Note that in this Hamiltonian regime, our system can be written as
\[
\begin{pmatrix}\eta_t  \\
u_t
\end{pmatrix}=\mathcal{J}_{bb}\mathcal{H}'\begin{pmatrix}\eta \\ \Phi \end{pmatrix}, \]
with
\[
\mathcal{J}_{bb}=\partial_x
\begin{pmatrix}0 &\left(I- b \partial^2_x\right)^{-1}\\ \left(I- b\partial^2_x\right)^{-1}&0
\end{pmatrix}.
\]
It is important to mention that if $(\eta(0),u(0))$ has an average zero, so does $(\eta(t),u(t))$ as long as the solution exists. Moreover, for a function $w\in L^2(\mathbb{R})$ having zero average on $\mathbb{R}$, we see that it is possible to define the operator $\partial_x^{-1}w$ as
\[
\partial_x^{-1}w(x) = \int_{-\infty}^x w(y)\,dy,
\]
in such a way that $\partial_x \partial_x^{-1} w=w$. On the other hand, there is a functional $\mathcal Q$ defined in $X$, known as the \textit{Charge}, which is conserved in time for classical solutions. This functional is given formally by \footnote{This holds by Noether's theorem \cite{Noe}.}
\begin{equation*}
\mathcal{Q}\begin{pmatrix}\eta \\ u \end{pmatrix}=-\frac12
\left<\mathcal{J}_{bb}^{-1}
\begin{pmatrix}\eta_x \\ u_x\end{pmatrix}, \begin{pmatrix}\eta \\ u \end{pmatrix}\right>=-\int_{\mathbb{R}}\big ( (I-b\partial_x^2)\eta\big )  u \,dx,
\end{equation*}
where $\left<\cdot,\cdot\right>$ denotes the standard $L^2$-inner product.

From this Hamiltonian structure, we have  that traveling waves of wave speed $\omega$ for the
generalized $abcb$-Boussinesq system \eqref{1bbl} correspond to stationary solutions of the modulated system
\[
\begin{pmatrix}\eta_t  \\
\Phi_t
\end{pmatrix}=\mathcal{J}_{bb}{\mathcal F_{\omega}}'\begin{pmatrix}\eta \\ \Phi \end{pmatrix},
\]
where
\begin{equation}\label{F_w}
\mathcal F_{\omega}(Y)=\mathcal H(Y)+\omega \mathcal Q(Y).
\end{equation}
In other words, they are the solutions to the system
\[
\mathcal H'(Y)+\omega\mathcal Q'(Y)=0.
\]
Now, let us provide some background on the stability issue.

\subsection{Historical background} Regarding the stability issue, Grillakis, Shatah, and Strauss \cite{GSS} gave a general framework to establish the stability of solitary waves for a class of abstract Hamiltonian systems, which will be called \textit{Grillakis-Shatah-Strauss (GSS) approach}. In this case, solitary-wave solutions $Y_\omega$  of the least energy are the minimum of a functional $\mathcal F_{\omega}$. In this approach, the analysis of the stability depends on the positiveness of the symmetric operator $ \mathcal F''_{\omega}(Y_{\omega})$ in a neighborhood of the solitary wave solution $Y_{\omega}$, and also  the strict convexity of the scalar function
$$
d_{1}(\omega)=\inf \{\mathcal F_{\omega}(Y): Y \in \mathcal M_{\omega}\},
$$
where $\mathcal M_{\omega}$ is a suitable set.

In this theory, one of the main tasks is to establish the positivity of $\mathcal F''_{\omega}(Y_{\omega})$. In one-dimensional spatial problems, the spectral analysis for the operator $ \mathcal F''_{\omega}(Y_{\omega})$ is reduced to studying the eigenvalues of an ordinary differential equation,  which becomes an ordinary differential equation with constant coefficients at $\pm\infty$ (for instance, see \cite{BSS} for more details). Based on the GGS approach, several works in the literature treat the stability of systems governed by partial differential equations.

For example, we can cite a series of works that show the stability of periodic waves for a dispersive system, such as a fifth-order KdV type equation, a nonlinear Klein–Gordon equation, a general class of nonlinear dispersive wave,  a fourth-order Schr\"odinger system, among others (see \cite{Natali3,Natali2,Angulo,Natali,Natali1} and the references therein for these cases). Additionally, there are recent results of stability/instability in models that arise in quantum field theory (for example, \cite{Palacios1,Palacios})\footnote{We mention that there is also interest in studying scattering and decay issues for the system \eqref{int_29e}. We suggest the nice articles \cite{KM2020,KMPP2019} and the reference therein.}.

Related to the $abcd$-Boussinesq model, several authors have studied this system. We mention first that, concerning explicit traveling-wave solutions, Chen \cite{Chen} has considered various cases for the $abcd$-Boussinesq system \eqref{int_29e}. She was able to write many traveling-wave solutions in the form $(\eta, u) =(\psi(x-\omega t), v(x-\omega t))$, depending on the constants $a,b,c$ and $d$. After that, adapting the positive operator theory of Krasnosell’skii \cite{Krasno1, Krasno2}, Bona and Chen \cite{BC2002} established the existence of traveling-wave solutions for the $abcd$ Boussinesq system \eqref{int_29e}, in the regime $$b, d>0,\ a, c\leq 0,\  |a|, |c| \leq \sqrt{bd}$$ and for $\omega>1$ such that
$$\omega^2> \max\left\{\frac{ac}{bd}, 1+ \frac{\sigma-\frac13}{b+d} \right\}.$$

More recently, stability issues have been treated in two works by Chen, Nguyen, and Sun.  In \cite{CNS2010}, the authors have shown that traveling-wave solutions of \eqref{int_29e} exist in the regime  $a + b + c +d < 0$, which corresponds to a large surface tension $\sigma>1/3$. In addition, they have also proven stability using techniques introduced earlier by Buffoni \cite{Buffoni} and Lions \cite{Lions,Lions2}.  Additionally, in \cite{CNS2011},  the authors considered the general case $b = d > 0$ and $a, c < 0,$ which, in particular, allows small surface tension cases. To be precise, they gave the existence of traveling-wave solutions in the presence of small propagation speeds, taking into account the coefficients satisfying
$$
a, c<0, \quad b=d, \quad|\omega|<\omega_0, \quad \omega_0:=\left\{\begin{array}{ll}
\min \{1, \frac{\sqrt{a c}}{b}\}, & b \neq 0, \\
1, & b=0 .
\end{array}\right.
$$

We also mention that considering a variation of \eqref{int_29e}, Hakkaev, Stanislavova, and Stefanov \cite{HSS2013}, showed the spectral stability of certain traveling-wave solutions for the Boussinesq ``abc" system, taking into advantage the explicit $\operatorname{sech}^2(x)$ like solutions of the form $(\eta, u) =(\psi(x-\omega t), v(x-\omega t))=(\psi, \text{const.}\psi)$, exhibited by Chen \cite{Chen}. In the article, they provided a complete rigorous characterization of the spectral stability in all cases for which $a=c<0, b>0$.

Finally, Loreno, Moraes, and Natali \cite{MLN2023} treated the stability of traveling-wave solutions for the $abcd$-Boussinesq model \eqref{int_29e}, considering the Hamiltonian regimes; however, in the periodic framework, which is completely different from our case.

Let us now briefly discuss the use of the GSS approach. To use this approach, in our work, the verification of the hypotheses of \cite[Theorem 3]{GSS} is difficult, since we do not have a close formula for traveling-wave solutions, making it almost impossible to compute $\mathcal F_{\omega}''(Y_{\omega})$ and $d''_{1}(\omega)$. However, we are still able to use the method by performing a direct approach to prove the stability of solitary-wave solutions of the system \eqref{1bbl}, using the characterization of $d_{1}(\omega)$ in terms of conservative quantities. This strategy was satisfactory in several cases, for example, as the pioneering work done by Shatah \cite{Shatah} in the case of the nonlinear Klein-Gordon equations, Bouard and Saut \cite{DBJC} for the KP equation, Liu and Wang \cite{YLXP} for the generalized KP equation, Levandosky \cite{L-98} concerning the fourth-order wave equation, Fukuizumi \cite{Fukuizumi} for the nonlinear Schr\"odinger equation with harmonic potential,  and Quintero \cite{JQ2005,JQ2010b} for the 2D Benney-Luke equation and the 2D Boussinesq type system, among others.

We mention that, since the literature in this area is vast, the cited references are a small sample - not exhaustive - about the stability results and the use of the GSS approach, thus, we suggest readers see more details in the previous works and the above-listed references, as well as the references therein.

\subsection{Main result} Given the state of the art, our work is motivated due to the results of \cite{CNS2010,CNS2011,HSS2013} that deal with one-dimensional $abcd$-Boussinesq system. We are now presenting our main result; however, first, let us introduce some notations.

By a solitary-wave solution, we shall mean a solution $(\eta, u)$ of \eqref{1bbl} taking the form
\begin{equation}\label{etaPhi}
\eta(x,t)=  \psi\left({x-\omega
t}\right), \ \ u(x, t)=
v\left({x-\omega t}\right),
\end{equation}
where $\omega$ denotes the wave's speed of propagation and $\psi, v$ approach to zero as $x$ goes to infinity. In what follows, we require that  $(\eta(x,t),u(x,t))\in X:=H^1(\mathbb{R})\times H^1(\mathbb{R})$ and restrict ourselves to the case (C1). Considering $\xi=x-\omega t$ and substituting the form of the solution \eqref{etaPhi} into \eqref{1bbl}, integrating once and evaluating the constants of the integration using the fact that $(\psi,v)\in X$, one sees that $(\psi,v)$ must satisfy the following system
\begin{equation}
\left\{
\begin{array}{ll}
-\omega\left(\psi-b\psi''\right) + v+a v''  +\psi
v^p& =0,
 \label{trav-eqs}\\
\omega\left(-v+bv''\right)+ \psi +c\psi''
+\frac{1}{p+1} v^{p+1}&=0.
\end{array}\right.
\end{equation}

We note that traveling-wave solutions can be considered as the critical points (i.e., minimizers) of a minimization problem, that is, the existence of solitary-wave solutions for the system \eqref{1bbl} is a consequence of a variational approach that applies a minimax type result since solutions $(\psi,v)$ of the system \eqref{trav-eqs} are the critical points of following functional $J_{\omega}=2\mathcal{F_{\omega}}$ defined by
\begin{equation}\label{JJ_w}
 J_{\omega}(\psi, v)=I_{\omega}(\psi, v)+ G(\psi, v),
 \end{equation}
 where $\mathcal{F_{\omega}}$ is given by \eqref{F_w}. Here,  the functional $I_\omega$ and $G$ are defined in the space $X$ by
\begin{equation}\label{I_w}
I_{\omega}(\psi,v)=I_1(\psi,v)+I_{2,{\omega}}(\psi,v),
\end{equation}
with
\begin{equation*}
I_1(\psi, v)=  \int_{\mathbb R}\left[\psi^2 - c(\psi')^{2} +
v^2-a( v')^2\right]dx,
\end{equation*}
\begin{equation}\label{NEW-R}
I_{2, \omega}(\psi, v)=  -2\omega\int_{\mathbb R}\left(\psi-b\psi'' \right)v\,dx= -2\omega\int_{\mathbb R}\left(\psi v+b\psi' v' \right)\,dx
\end{equation}
and
\begin{equation*}
G(\psi, v) =   \frac{2}{p+1}\int_{\mathbb R}\psi v^{p+1}\,dx.
\end{equation*}
\medskip

\begin{rem}\label{rems} Some remarks are worthy of mentioning.
\vspace{0.2cm}
\begin{itemize}
\item[a.] A ground state solution is a solitary-wave solution that minimizes the action functional $J_\omega$ among all the nonzero solutions of \eqref{trav-eqs}.
\item[b.]  If $(\psi, v)$ is a solution of \eqref{trav-eqs},  the following quantities hold,
\begin{align}\label{critical1}
J_{\omega}(\psi, v)& = \frac{p}{p+2}I_\omega(\psi, v),\\
J_{\omega}(\psi, v)& = -\frac{p}{2}G(\psi, v),\label{critical2}\\
I_\omega(\psi, v)& = -\frac{p+2}{2}G(\psi, v). \label{critical3}
\end{align}
\item[c.]  So, as proposed in \cite[Theorems 2 and 3]{GSS}, the analysis of the orbital stability of ground state solutions depends upon some properties of the scalar function $d:=d(\omega)$ given by
\begin{equation}\label{d_w-a}
d(\omega)= \inf\{J_\omega(\psi, v) :   \   (\psi, v)\in\mathcal{M}_\omega\},
\end{equation}
where $\mathcal{M}_\omega$ is defined by the following set
\begin{equation}\label{M_w}
\mathcal{M}_\omega=\left\{ (\psi, v)\in X  \  :  \  K_\omega(\psi, v)=0,  \
(\psi, v)\neq0\right\},
\end{equation}
with $$
K_\omega(\psi, v)=\left<J_{\omega}'(\psi, v), (\psi, v)\right>.
$$
\end{itemize}
\end{rem}

With all these notations and definitions in hand, the main result of the article gives a positive answer to the \textbf{orbital stability problem} (or actually the \textbf{set stability}) presented at the beginning of the introduction for certain $p \geq 1$. In other words, the generalized $abcb$-Boussinesq system \eqref{1bbl} has a set of traveling-wave solutions that is stable when the \textit{wave speed} $\omega_0$ of the traveling waves is near $1^-$.

\begin{thm}[\textit{Set stability}]\label{main} For the generalized $abcb$-Boussinesq system \eqref{1bbl}, with $b>0$, $a<0$ and $c<0$, and $p$ satisfying \eqref{assump_p}, there is a non-empty set of traveling-wave solutions with speed $\omega$, denoted by $\tilde {\mathcal G}_{\omega}$,   if
\begin{equation}\label{relation_c}
2b < -a -c  \quad \text{and}\quad 0< |\omega| <\min\left\{1,\sqrt{ac}/b \right\},
\end{equation}
are satisfied.  Furthermore, if $b \leq \sqrt{ac}$ \footnote{This implies $2b<  -a-c$ if $a \not= c$.}, $2b  <  - a - c$ and $ 1\leq p < p_0 $  with a unique critical number $p_0 > 4 $, the set $\tilde {\mathcal G}_{\omega}$ for the generalized $abcb$-Boussinesq system \eqref{1bbl}  with $\omega >0$, but near $1^{-}$, is stable in the sense of \eqref{OSP} with the set being $\tilde {\mathcal G}_{\omega}$. In other words,
given $(\tilde \eta_{\omega},\tilde u_{\omega}) \in \tilde {\mathcal G}_{\omega}$ for $\omega >0$, but near $1^{-}$, if $(\eta(0),u(0))$ is near $(\tilde \eta_{\omega},\tilde u_{\omega})$ in the space $X$,  the solution $(\eta(t),u(t))$ remains near the set $\tilde {\mathcal G}_{\omega}$ in the space X.
\end{thm}
\begin{rem} \label {Remark 1.3}
Here, it is very interesting to see that the stability of the set $\tilde {\mathcal G}_{\omega}$ is obtained for any $p$ with  $ 1\leq  p < p_0$ where $p_0 $ is strictly large than $4$, while the critical exponent for the stability of solitary waves for the generalized KdV equations is $p_0 = 4$. Moreover, it can be shown numerically that $p_0$ is approximately equal to $4.2280673976$. The instability problem for $ p  > p_0$ will be studied in the near future.
\end{rem}

\subsection{Heuristic and structure of the article}  Let us highlight the present work's contribution and provide a summary of how Theorem \ref{main} can be obtained.

Observe that the natural space (energy space) in which we consider the well-posedness of the Cauchy problem is $X$. This comes from the fact that the Hamiltonian structure defined in \eqref{ham} and $\mathcal{F}_{{\omega}}$ given by \eqref{F_w} require $(\eta(x,t),u(x,t))\in X$ to be well defined. Additionally, these conditions already characterize the natural space (energy space) for traveling-wave solutions of the generalized  $abcb$-Boussinesq system \eqref{1bbl}.

The difficulty in using \cite[Theorem 3]{GSS} appears when computing $\mathcal{F}''_{\omega}$ around the traveling wave $(\eta_{\omega},u_{\omega})$ since we do not know explicitly the characterization of this pair for the generalized $abcb$-Boussinesq  model. In other words, it is almost impossible to establish the spectral hypotheses on the second variation of the action functional on the traveling wave. We appeal to the variational characterization of traveling-wave solutions to overcome this difficulty. Precisely, by the quantities \eqref{critical1} and \eqref{critical3}, we can define a scalar function $d(\omega)$, see equation \eqref{d_w}, establishing the convexity of $d$, since we can prove that $d''(\omega)>0$.

Two tools will be useful to prove the minimization problem and show that $d(\omega)$ is strictly convex. The first one is related to the existence of traveling-wave solutions for \eqref{trav-eqs} as a minimizer problem. In our context, we will invoke the classical Lions' concentration-compactness Theorem \cite{Lions,Lions2}. Together with this result, the second tool is to see that the generalized Korteweg-de Vries (KdV) equation
\begin{equation}\label{kdv}
u_t + u_{x}+\left(\frac13-\sigma\right)u_{xxx}+(u^{p+1})_{x}=0,
\end{equation}
emerges from the generalized $abcb$-Boussinesq system (up to some order). For this fact, it is natural to expect that the family of solitary-wave solutions of the generalized $abcb$-system \eqref{1bbl} converges to nontrivial solitary-wave solutions of the generalized KdV equation \eqref{kdv}. Putting these two important tools together,  we can reach the convexity of the scalar function $d(\omega)$, taking into account an important fact of a transformed system related to \eqref{trav-eqs} (see Appendix \ref{A}). Summarizing what concerns our main result, Theorem \ref{main}, the following points are worthy of mentioning.
\vspace{0.2cm}
\begin{itemize}
\item[a.] In \cite{HSS2013}, the authors suggest that the GSS approach fails when applied to the system \eqref{1bbl1}. However, our work showed that the stability (Theorem \ref{main}) is a direct consequence of the GSS approach. The main ingredients in this analysis are: \textit{KdV scaling for the generalized $abcb$-Boussinesq system and its properties} and  \textit{GSS approach}.
\vspace{0.1cm}
\item[b.] To the authors' best knowledge, no attempt has been made to apply this strategy for the system \eqref{1bbl}. Thus, in the context presented in this article, we give a necessary first step in understanding the stability using the previous ingredients for the generalized $abcd$-Boussinesq system in the Hamiltonian case.
\vspace{0.1cm}
\item[c.] It is important to point out that our main result, Theorem \ref{main}, suggests that the set $\tilde {\mathcal G}_{\omega}$ for the generalized $abcb$-Boussinesq system \eqref{1bbl}  with speed $\omega >0$, but near $1^{-}$, could be is unstable when $p > p_0$, i.e, relation \eqref{OSP} fails. In this way, we will soon present a detailed study of the instability of the generalized $abcb$-Boussinesq system in a forthcoming paper.
\item[d.]Note that the solutions of the system \eqref{1bbl}, for $p>1$ with the initial conditions close to travelling waves with small amplitudes, are global (for the case $p=1$,
the global  well-posedness problem of \eqref{1bbl} was addressed in \cite{BCS04}). The global existence follows from the presence of invariant sets preserved by the flow. A brief proof of the global existence is provided in Appendix \ref{C}. 
\end{itemize}


We finish this introduction with an outline of this work, which consists of six parts, including the introduction. Section \ref{sec2} gives a brief discussion of the existence of minimizers, that is, we present the existence of solitary-wave solutions for the system \eqref{1bbl}. Section \ref{sec3} is devoted to proving carefully the inter-relation between the generalized KdV equation \eqref{kdv} and the generalized $abcb$-Boussinesq system \eqref{1bbl}. Section \ref{sec4} gives the properties of the scalar function $d(\omega)$, in particular, the strict convexity of $d(\omega)$ for $\omega \in(0,1)$, near $1^-$. In Section \ref{sec5}, we will give the proof of Theorem \ref{main} using the GSS approach, showing that the solution set of the generalized $abcb$-Boussinesq system \eqref{1bbl} is stable. Finally, Appendix \ref{A} is devoted to giving properties of a transformed system associated with \eqref{trav-eqs}, which is the key point to prove the convexity for the scalar function $d(\omega)$. Moreover, in Appendix \ref{B} we presented the concentration-compactness argument which is used in Section \ref{sec3}. Appendix \ref{C} provides brief proofs of the global existence of solutions for small initial data and the existence of a branch of minimizers passing through a global minimizer.


\section{Brief discussion on the existence of minimizers} \label{sec2}

It is well known that the existence of traveling-wave solutions for \eqref{trav-eqs} as a minimizer of the following problem
\begin{equation}\label{mp}
\mathcal I_{\omega}=\inf \left\{I_{\omega}(\psi, v) \in X: G(\psi, v)=-1 \right\}
\end{equation}
is based on the existence of a compact embedding (local) result and also on an important result by P.-L. Lions, which completely characterizes the convergence of measures, is known as the \textit{Concentration-Compactness principle}.

\begin{thm}[P.-L. Lions \cite{Lions,Lions2}]\label{lions}
Suppose $\left\{\nu_n\right\}$ is a sequence of nonnegative measures on $\mathbb{R}^k$ such that
$$
\lim _{n \rightarrow \infty} \int_{\mathbb{R}^k} d \nu_n=\mathcal{I}\, .
$$
Then, there is a subsequence of $\left\{\nu_n\right\}$ (which is denoted the same) that satisfies only one of the following properties.
\vspace{0.2cm}
\begin{itemize}
\item[i.]\textbf{Vanishing:} For any $R>0$,
$$
\lim _{n \rightarrow \infty}\left(\sup _{x \in \mathbb{R}^k} \int_{B_R(x)} d \nu_n\right)=0,
$$
where $B_R(x)$ is the ball of radius $R$ centered at $x$.
\vspace{0.2cm}
\item[ii.]\textbf{Dichotomy:} There exist $\theta \in(0, \mathcal{I})$ such that for any $\gamma>0$, there are $R>0$ and a sequence $\left\{ x_n \right\}$ in $\mathbb{R}^k$ with the following property: Given $R^{\prime}>R$ there are nonnegative measures $\nu_n^1, \nu_n^2$ such that
\begin{enumerate}
\item[a)] $0 \leq \nu_n^1+\nu_n^2 \leq \nu_n$;
\item[b)] $\operatorname{supp}\left(\nu_n^1\right) \subset B_R\left(x_n\right), \quad \operatorname{supp}\left(\nu_n^2\right) \subset \mathbb{R}^k \backslash B_{R^{\prime}}\left(x_n\right)$;
\item[c)] $$\lim \sup _{n \rightarrow \infty}\left(\left|\theta-\int_{\mathbb{R}^k} d \nu_n^1\right|+\left|(\mathcal{I}-\theta)-\int_{\mathbb{R}^k} d \nu_n^2\right|\right) \leq \gamma.$$
\end{enumerate}
\vspace{0.2cm}
\item[iii.]\textbf{Compactness:} There exists a sequence $\left\{ x_n \right\}$ in $\mathbb{R}^k$ such that for any $\gamma>0$, there is $R>0$ with the property that
$$
\int_{B_R\left(x_n\right)} d \nu_n \geq \mathcal{I}-\gamma, \quad \text { for all } n.
$$
\end{itemize}
\end{thm}

To apply this result to our case, we note that for a
minimizing sequence $\left\{\left(\psi_n, v_n\right)\right\}$, we may define the concentration function induced by the integrand of $I_\omega(\psi,v)$ as $$\rho_n=\left(\psi_n^{\prime}\right)^2+\psi_n^2+\left(v_n^{\prime}\right)^2+v_n^2,$$
and the measure
\[
\nu_n(A)=\int_A \rho_n(x)\,dx\leq \left\|\left(\psi_n, v_n\right)\right\|_X \leq C, \ \ \mbox{for all $n\in \mathbb{N}$}\, ,
\]
where $I_\omega ( \psi_n , v_n ) $ is equivalent to $\int_{\mathbb{R}} \rho_n dx$ if $ 0 < |\omega | < \min ( 1 , \sqrt{ca}/b ) $, with $A\subset\mathbb{R}$.
As $\left\|\left(\psi_n, v_n\right)\right\|_X \leq C$ for all $n$, we can extract a convergent subsequence which we again denote as $\left\{\left(\psi_n, v_n\right)\right\}$, so that
$$
\lambda=\lim _{n \rightarrow \infty} \int_{-\infty}^{\infty} \rho_n(x) d x
$$
exists. Define a sequence of non-decreasing functions $M_n:[0, \infty) \rightarrow[0, \lambda]$ as follows:
$$
M_n(r)=\sup _{y \in \mathbb{R}} \int_{y-r}^{y+r} \rho_n(x) d x .
$$
Since $M_n(r)$ is a uniformly bounded sequence of non-decreasing functions in $r$, one can show that it has a subsequence, which we still denote as $M_n$, that converges pointwise to a non-decreasing limit function $M(r):[0, \infty) \rightarrow[0, \lambda]$. Let
$$
\lambda_0=\lim _{r \rightarrow \infty} M(r): \equiv \lim _{r \rightarrow \infty} \lim _{n \rightarrow \infty} \sup _{y \in \mathbb{R}} \int_{y-r}^{y+r} \rho_n(x) d x .
$$
Then $0 \leq \lambda_0 \leq \lambda$.

As is well known for dispersive type systems (see, for instance, \cite{CNS2010,DPDE2013}, for one and two-dimensional cases, respectively),  ruling out vanishing and dichotomy for a minimizing sequence of $I_{\omega}$, the Lion’s Concentration Compactness Theorem \ref{lions} ensures the existence of a subsequence of $\left\{\nu_n\right\}$ satisfying the compactness conditions. Therefore, as a consequence of local compact embedding, the minimizing sequence $\left\{\left(\psi_n, v_n\right)\right\}$ (or a subsequence) is compact in $X$,  up to translation. The proof is very standard and will be omitted. Thus,  the following theorem holds for the generalized $abcb$-Boussinesq system \eqref{1bbl}.
\begin{thm}\label{pt} Let $ 0 < |\omega | < \min ( 1 , \sqrt{ca}/b ) $.
If $\{(\psi_n,v_n)\}$ is a minimizing sequence for \eqref{mp}, then there is a subsequence, still denoted by the same index, a sequence of points $y_n \in \mathbb{R}$, and a minimizer $(\psi_0,v_0)\in
X$ of \eqref{mp}, such that the translated functions
$$(\tilde{\psi}_n,\tilde{v}_n)=(\psi_n(\cdot+y_n),v_n(\cdot+y_n))\to (\psi_0,v_0)\quad \text{strongly in } X.$$
\end{thm}

\subsection{Minimization problem} With the previous result in hand, let us prove that \eqref{trav-eqs} has a nontrivial solution. Considering the minimization problem \eqref{mp}, observe that the constraint $G(\psi, v)=-1$ is necessary since the quantity given by \eqref{critical3} needs to be positive. Moreover,  noting that $H^{1}(\mathbb R) \hookrightarrow L^{q}(\mathbb R)$ for all $q\geq 2$  and the Young's inequality,  we have
\begin{equation}\label{G1G2}
|G(\psi,v)|\leq M(\|\psi\|^{p+2}_{L^{p+2}(\mathbb
R)}+\|v\|^{p+2}_{L^{p+2}(\mathbb R)})\leq M \|(\psi, v)\|_{X}^{p+2}.
\end{equation}
Our first lemma ensures some boundedness for the quantity \eqref{critical3}.
\begin{lem}
For \eqref{relation_c} being satisfied, the functional $I_{\omega}$ defined by \eqref{I_w} is nonnegative. Moreover,  there are positive constants
$M_1(a, b, c, \omega)$ and $M_2(a, b, c, \omega)$ such that
\begin{equation}\label{IG1}
M_1 \|(\psi, v)\|_{X}^{2} \leq I_\omega(\psi, v)\leq M_2
 \|(\psi, v)\|_{X}^{2},
\end{equation}
and $\mathcal I_{\omega}$, given by \eqref{mp},  is finite and positive.
\end{lem}
\begin{proof} In fact,  using the quantity \eqref{I_w} and Young's inequality, we
obtain that
\begin{equation}\label{eqIN}
\begin{split}
I_\omega(\psi,v)\leq&\int_{\mathbb R}
\left[(1+|\omega|)\psi^2+\left(|c|+b|\omega| \right)(\psi')^2\right.\\ &\left.\quad\quad\quad+(1+|\omega|)v^2+\left(|a|+b|\omega|
\right)(v')^2
\right]dx \\
\leq& \max\left(1+|\omega|,|c|+b|\omega| , |a|+b|\omega|\right)\|(\psi,v)\|^2_{X}
\end{split}
\end{equation}
and
\begin{equation}\label{eqINa}
\begin{split}
I_\omega(\psi,v) =& \int_{\mathbb R}
\bigg [ (\psi - \omega v )^2 + \left ( \sqrt{|c|} \psi' - \left (b\omega /\sqrt{|c|} \right ) v' \right )^2 +
(1-|\omega|^2)\psi^2 \\
& \qquad +\left(|a|-\left ( b^2\omega^2/|c|\right )\right )   (v')^2
\bigg ]dx \\ \geq &\ C_0\|(\psi,v)\|^2_{X}\, ,
\end{split}
\end{equation}
if $ 0 < |\omega | < \min ( 1 , \sqrt{ca}/b ) $.
Inequalities \eqref{eqIN} and \eqref{eqINa} give \eqref{IG1}. On the other hand, using that $G(\psi,v)=-1$, we have from \eqref{G1G2} that
$$
M_1\left(I_{\omega}(\psi, v)\right)^{\frac{p+2}{2}} \geq M||(\psi, v)||^{p+2}_{X}\geq |G(\psi, v)| =1 \, ,
$$
which implies
\[
I_{\omega}(\psi, v)\geq  M_1^{-\frac{2}{p+2}},
\]
meaning that the infimum $\mathcal{I}_{\omega}$ is finite and positive.
\end{proof}

Thanks to Theorem \ref{pt},  the problem \eqref{mp} has a minimizer.  
Therefore,  the main result in this section ensures that \eqref{trav-eqs} has a nontrivial solution.
\begin{rem}\label{yyyyyy}
The problem \eqref{mp} gives global minimizers since $(\psi, v)$ does not have other restrictions except for the one stated. By using the same procedure, we may restrict $(\psi, v)$ in some open set $\Omega$ in $H^1\times H^1$ and  consider 
\begin{equation}
\hat{\mathcal{I} }_\omega = \inf \{ I _\omega (\psi , v) \ : \ (\psi , v)\in \Omega\, , \  \ G(\psi , v ) = -1\}\, .\label{wwwwww}
\end{equation}
If there is a $(\psi_0, v_0) \in \Omega$ such that $G(\psi_0 , v_0 ) = -1 $ and $I _\omega (\psi _0, v_0) < \inf \{ I _\omega (\psi , v) \ : \  (\psi , v) \in \partial \Omega\, , \  \ G(\psi , v ) = -1\} $, then the problem \eqref{wwwwww} has a minimizer in $\Omega$. Such minimizers may be called local minimizers.
\end{rem}

\begin{thm}\label{wsol}
Let $(\psi_0,v_0)$  be a minimizer (or a local minimizer) for the problem \eqref{mp}. Then, the function $(\psi,v)=\beta(\psi_0,v_0)$ is a nontrivial solution of \eqref{trav-eqs} for $\beta= (-\lambda)^{\frac1{p}}$ with $\lambda=-\frac{2}{p+2}\mathcal I_{\omega}$.
\end{thm}
\begin{proof}
From the Lagrange multiplier theorem, there exists $\lambda$ such that $$I_{\omega}'(\psi_0, v_0)= \lambda G'(\psi_0, v_0).$$ On the other hand,
\begin{equation*}
\begin{split}
2\mathcal I_{\omega}=& \ 2 I_{\omega}(\psi_0, v_0)\\=& \left< I'_{\omega}(\psi_0, v_0), (\psi_0, v_0)\right>\\=&\ \lambda\left< G'(\psi_0, v_0), (\psi_0, v_0)\right>\\=&\  \lambda(p+2)G(\psi_0, v_0).
\end{split}
\end{equation*}
In this case, we have that the Lagrange multiplier $\lambda$ is given by  $\lambda=-\frac{2}{p+2}\mathcal I_{\omega}$. If we take $(\psi,v)=\beta(\psi_0,v_0)$ with  $\beta= (-\lambda)^{\frac1{p}}$, we see that
\[
I_{\omega}'(\psi, v) + G'(\psi, v)=0 \ \ \iff \ \ \lambda +\beta^p=0,
\]
showing the result.
\end{proof}
\begin{defn}
We will now call the solution given in the previous theorem as \textit{solitary-wave solution}. This solution is indeed a classical solution of \eqref{trav-eqs}. We will also denote the set of those traveling-wave solutions with speed $\omega$ obtained from the minimizers of \eqref{mp} by $\tilde {\mathcal{G}}_\omega$.
\end{defn}
We remark that  \eqref{trav-eqs} is invariant under that the reflection $x \rightarrow -x$, which may imply that the solutions may be even in $x$. On the other hand, under certain conditions, Bona and Chen \cite{BC2002} established the existence of even traveling wave solutions of \eqref{int_29e} by employing the method in \cite{Krasno1, Krasno2} and assuming the solutions are inherently even by construction.


\section{The KdV scaling for the generalized $abcb$-Boussinesq  system}\label{sec3}
In this section, we present some auxiliary lemmas that are paramount to proving the main result of this article.  We will see that a renormalized family of solitary-wave solutions of the generalized $abcb$-Boussinesq system converges to nontrivial solitary-wave solutions for the generalized KdV equation, assuming the speed velocity $\omega$ close to $1^-$ as $\ep \to 0^+$ with $ b \leq \sqrt{ac}$ and balancing the effects of nonlinearity and dispersion\footnote{This phenomenon was characterized also for solitary-wave solutions of $2D$ Boussinesq-Benney-Luke system in  \cite{DPDE2013}, where the authors used the characterization of solitary-wave solutions for the (KP-I) model given in \cite{DBJCb}.}.

Set $\epsilon>0$, $\omega^{2}=1- \epsilon^{\frac{2}{p+1}}$ and, for a given couple $(\psi,v)\in X$,  consider the following scaling 
\begin{equation}\label{defuz}
\psi(x)=\epsilon^{\frac1{(p+1)(p+2)}}z(y), \ \ \ v(x)=\epsilon^{\frac1{(p+1)(p+2)}}w(y) \ \ \  \text{with} \
y=\epsilon^{\frac1{p+1}}x.
\end{equation}
Now, define the following quantities,
\begin{equation}\label{i1}
I^\epsilon(z,w)=I^{1,\epsilon}(z,w)+I^{2,\epsilon}(z,w),
\end{equation}
and
\begin{equation}\label{g_a}
G(z,w)=\frac{2}{p+1}\int_{\mathbb{R}}z w^{p+1}\,dy.
\end{equation}
Here
\begin{equation}\label{i1e}
I^{1,\epsilon}(z,w)=\int_{\mathbb{R}}\left(\epsilon^{-\frac2{p+1}}z^2- c (z')^2+ \epsilon^{-\frac2{p+1}}w^2- a (w')^2\right)dy
\end{equation}
and
\begin{equation}\label{i2e}
I^{2,\epsilon}(z,w)=-2 \omega \int_{\mathbb{R}}(\epsilon^{-\frac2{p+1}}zw+b z' w')\,dy.
\end{equation}

Straightforward calculations give us the following relations:
\begin{equation*}
I_1(\psi,v)=\epsilon^{\frac{p+4}{(p+1)(p+2)}}I^{1,\epsilon}(z,w),
\end{equation*}
\begin{equation*}
I_{2,\omega}(\psi,v)=\epsilon^{\frac{p+4}{(p+1)(p+2)}}I^{2,\epsilon}(z,w),
\end{equation*}
\begin{equation*}
I_{\omega}(\psi,v)=\epsilon^{\frac{p+4}{(p+1)(p+2)}}I^\epsilon(z,w),
\end{equation*}
and
\begin{equation*}
G(\psi ,v)=G^{\epsilon}(z,w)=G(z,w),
\end{equation*}
where $I^{1,\epsilon}$, $I^{2,\epsilon}$, $I^\epsilon$ and $G$ are given by \eqref{i1e}, \eqref{i2e}, \eqref{i1} and \eqref{g_a},  respectively.

Under relations \eqref{relation_c}, there exists a family
$\left\{(\psi_{\omega},v_{\omega})\right\}_{\omega}\,$ such that
$$
I_{\omega}(\psi_{\omega},v_{\omega})=\mathcal I_{\omega},  \  \  G(\psi_{\omega},v_{\omega})=-1.  \ \
$$
Then, if we denote
$$\mathcal I^\epsilon:=\inf\left\{I^\epsilon(z,w) \   :    \  (z,w)\in X\,\,\,\,
\text{with} \ \  \  G(z,w)=-1 \right\},
$$
there is a corresponding family
$\{(z^\epsilon,w^\epsilon)\}_\epsilon$ such that
\begin{equation*}
\mathcal I^\epsilon=I^\epsilon(z^\epsilon,w^\epsilon), \ \  \
G(z^\epsilon,w^\epsilon)=-1,
 \ \  \  \mathcal I_{\omega}=\epsilon^{\frac{p+4}{(p+1)(p+2)}}\mathcal I^\epsilon\,.
\end{equation*}
We also have that  $(z^\epsilon,w^\epsilon)$ is a solution of the system
\begin{equation}
\left\{
\begin{array}{rl}
\epsilon^{-\frac2{p+1}}(w-\omega z)+\omega bz''+aw''+\left(\frac{2}{p+2}\right) \mathcal I^{\ep}  z
w^{p}& =0, \\
\\ \label{trav-ep}
\epsilon^{-\frac2{p+1}}(z-\omega w)+\omega bw''+cz''+\left(\frac2{(p+1)(p+2)}\right)  \mathcal I^{\ep} w^{p+1}&=0.
\end{array}\right.
\end{equation}

Now, let us define in $X$ the following two functionals
\begin{equation}\label{Je}
J^{\ep}(w)= I^{\ep}(\omega w, w):=\int_{\mathbb{R}}\epsilon^{-\frac2{p+1}}(1-\omega^{2})w^{2}dy
+\int_{\mathbb{R}}-\left((2b+c)\omega^{2}+a\right)(w')^{2}dy
\end{equation}
and
\begin{equation}\label{We}
K^{\ep}(w)= G^{\epsilon}(\omega w, w).
\end{equation}
Define the following number $\J^{\ep}$
\begin{equation}\label{JJe}
\J^{\ep}= \inf \{J^{\ep}(w): w \in H^1(\R), \ K^{\ep}(w)=-1 \}\, ,
\end{equation}
where $ {\mathcal I^\epsilon} \leq \J^{\ep} $, and need to keep in mind the following quantity defined by \eqref{i1}:
\begin{equation}\label{I-ep}
\begin{split}
I^{\ep}(z, w)=& \int_{\mathbb R}\left(\epsilon^{-\frac2{p+1}}(z-\omega(\ep) w)^2+(1-\omega^2(\ep))\epsilon^{-\frac2{p+1}}w^2 \right)dy \\
&+\int_{\mathbb R}\left(|c|\left(z'  -  \frac{b\omega(\ep)}{|c|} w'\right)^2 +\left(\frac{ac-b^2 \omega^2(\epsilon)}{|c|}\right) (w')^2\right)\,dy.
\end{split}
\end{equation}
Let us give some behavior, as $\ep \to 0$, for the functional \eqref{We} and the number \eqref{JJe}.
\begin{lem}\label{limite} Considering the functionals \eqref{Je} and \eqref{We}, it follows that
\begin{equation}\label{lim}
\lim_{\epsilon\rightarrow0^+}K^\epsilon(w^\epsilon)=\lim_{\epsilon\rightarrow0^+}G^{\ep}(\omega(\ep) w^\epsilon, w^\epsilon)=-1
\end{equation}
and
\begin{equation}\label{lim_2}
\lim_{\epsilon\rightarrow0^+}\mathcal
I^\epsilon=\lim_{\epsilon\rightarrow0^+}J^\epsilon(w^\epsilon)=\mathcal J^0>0\, ,
\end{equation}
where
\begin{equation*}
\begin{split}
\mathcal J^{0} &= \inf \{J^0(w): w \in H^1(\R), \ K^{0}(w)=-1 \},\\
J^{0}(w) &=\int_{{\mathbb{R}}} \left(w^2 + \left(\sigma - \tfrac{1}{3}\right)w_{x}^2 \right)dy,  \\
K^{0}(w) & =  \frac 2 {p+1} \int_{{\mathbb{R}}} w^{p+2} dy.
\end{split}
\end{equation*}
\end{lem}
\begin{proof}
Let $v\in H^1(\R)$ satisfy $K^0(v)=-1$ and define $$\alpha=(\omega(\ep))^{-\frac1{p+2}}\, .$$ Then, for such a $v$,  we have that $K^{\ep}(\alpha v)=-1$ and thus,
\begin{equation}\label{CQS}
J^{\ep}(\alpha v)=\alpha^2 J^{\ep}(v)\geq \mathcal J^{\ep}.
\end{equation}
Now, we note that $\lim_{\ep \to 0^+} |\alpha|=1$.  On the other hand, using $\omega^{2}=1- \epsilon^{\frac{2}{p+1}}$ and that $c+a+2b=\frac13-\sigma$, we conclude the following
\begin{equation*}
\begin{split}
\lim_{\ep \to 0^{+}}J^{\ep}(v)=&\lim_{\ep \to 0^{+}}\int_{\mathbb{R}}\epsilon^{-\frac2{p+1}}(1-\omega^{2})v^{2}dy+\lim_{\ep \to 0^{+}}\int_{{\mathbb{R}}} \left((-c-2b)\omega(\ep)^2-a\right)(v')^2\,dy\\
= &\int_{\R}(v^2 + \left(\sigma-\frac13\right) (v')^2)\,dy = J^0(v).
\end{split}
\end{equation*}
Consequently, \eqref{CQS} implies that
\begin{equation}\label{RR}
\mathcal J^0 \geq \limsup_{\ep \to 0} \mathcal J^{\ep} \quad \text{and} \quad \mathcal{J}^0\geq \limsup _{\epsilon \rightarrow 0^{+}} \mathcal{I}^\epsilon .
\end{equation}
Now, observe that
$$
\lim _{\epsilon \rightarrow 0^{+}} K^0\left(w^\epsilon\right)  =\lim _{\epsilon \rightarrow 0^{+}}\frac 2 {p+1}  \omega \int_{\mathbb{R}}\left(w^\epsilon\right)^{p+2} dy =\lim _{\epsilon \rightarrow 0^{+}} G^\epsilon\left(\omega w^\epsilon, w^\epsilon\right).
$$
We claim that
$$
\lim _{\epsilon \rightarrow 0^{+}} K^0\left(w^\epsilon\right)=-1
$$
To prove it, we need to show that
\begin{equation}\label{R1}
\lim _{\epsilon \rightarrow 0^{+}} G^\epsilon\left(z^\epsilon, w^\epsilon\right)=\lim _{\epsilon \rightarrow 0^{+}} G^\epsilon\left(\omega w^\epsilon, w^\epsilon\right),
\end{equation}
since  $$\lim _{\epsilon \rightarrow 0^{+}} G^\epsilon\left(z^\epsilon, w^\epsilon\right)=-1.$$
Note that \eqref{R1} is equivalent to prove that
\begin{equation}\label{R2}
\lim _{\epsilon \rightarrow 0^{+}}\left|\int_{\mathbb{R}}(z^{\epsilon}-\omega(\epsilon)w^{\epsilon})(w^{\epsilon})^{p+1}dy\right|=0.
\end{equation}
To show \eqref{R2}, we note that
\begin{equation*}
\begin{split}
\left|\int_{\mathbb{R}}(z^{\epsilon}-\omega(\epsilon)w^{\epsilon})(w^{\epsilon})^{p+1}dy\right|\leq &C\left\|z^{\epsilon}-\omega(\epsilon)w^{\epsilon}\right\|_{L^{2}(\mathbb{R})}\left\|w^{\epsilon}\right\|^{p+1}_{H^{1}(\mathbb{R})}\, ,
\end{split}
\end{equation*}
since \eqref{I-ep} together with $b^2 < ac $ implies that
\begin{equation}
\begin{cases}\label{ccccc}
 &\left\|z^{\epsilon}-\omega(\epsilon)w^{\epsilon}\right\|_{L^2(\mathbb{R})}=O(\epsilon^{\frac{1}{p+1}}),\quad \text{and}\\ &\left \| w^{\epsilon}\right\|_{H^{1}(\mathbb{R})}\, ,\left \| z^{\epsilon}\right\|_{H^{1}(\mathbb{R})}\  \mbox{are uniformly bounded}\, ,
 \end{cases}
\end{equation}
\color{black}which ensures that \eqref{R2} holds when $\epsilon \rightarrow 0^{+}$. Thus, we conclude that
$$
\lim _{\epsilon \rightarrow 0^{+}} K^0\left(w^\epsilon\right)=\lim _{\epsilon \rightarrow 0^{+}} K^\epsilon\left(w^\epsilon\right)=\lim _{\epsilon \rightarrow 0^{+}} G^\epsilon\left(z^\epsilon, w^\epsilon\right)=-1,
$$
showing \eqref{lim}. If $\epsilon$ is small enough, it is obtained that $K^0\left(w^\epsilon\right) \neq 0$, and
$$
\mathcal{J}^0 \leq J^0\left(\frac{w^\epsilon}{-K^0\left(w^\epsilon\right)^{\frac{1}{p+2}}}\right)=\frac{J^0\left(w^\epsilon\right)}{|K^0\left(w^\epsilon\right) | ^{-\frac{2}{ p+2}}}\, ,
$$
together with
$$J^\epsilon\left(w^\epsilon\right)-J^0\left(w^\epsilon\right)=o(1)\, .$$
Next, we note that  \eqref{I-ep} is
\begin{equation*}
\begin{split}
{\mathcal{I}}^\epsilon = &I^{\ep} (z^\ep , w^\ep )\\=& \int_{\mathbb R}\left(\epsilon^{-\frac2{p+1}}\left (z^\ep-\omega(\ep) w^\ep\right )^2+ \left ( w^\ep\right )^2 \right)dy \\
&+\int_{\mathbb R}\left(|c|\left(\left (z^\ep\right ) '  -  \frac{b\omega(\ep)}{|c|} \left ( w^\ep\right ) '\right)^2 +\left(\frac{ac-b^2 \omega^2(\epsilon)}{|c|}\right) \left (\left ( w^\ep\right )'\right )^2\right)\,dy\\
 \leq & J^\epsilon \left ( w^\epsilon \right ),
 \end{split}
\end{equation*}
where $\limsup_{\ep \rightarrow 0^+} {\mathcal{ I}}^\epsilon \leq {\mathcal{J}}^0  \  \mbox{and} \ \ G(z^\ep , w^\ep ) = -1$.

Now, we consider $\liminf_{\ep \rightarrow 0^+} {\mathcal{I}}^\epsilon $, which gives that there is a sequence $\{\ep_j\}\rightarrow 0^+ $ such that
$$
\lim_{j \rightarrow\infty } {\mathcal{I}}^{\epsilon_j} = \lim_{j \rightarrow\infty } I^{\ep_j} \left (z^{\ep_j} , w^{\ep_j} \right ) = \liminf_{\ep \rightarrow 0^+} {\mathcal{I}}^\epsilon \leq {\mathcal{J}}^0 \quad \mbox{with} \quad \ G\left (z^{\ep_j} , w^{\ep_j} \right ) = -1.
$$
For this sequence $\left (z^{\ep_j} , w^{\ep_j} \right )$, the following claim holds.
\smallskip

\noindent{\bf Claim I:} {\it The sequence of minimizers $\left (z^{\ep_j} , w^{\ep_j} \right )$ for $I^{\ep_j} (z , w)$ with $G\left (z , w \right ) = -1$ has a subsequence of $\left (z^{\ep_j} , w^{\ep_j} \right )$ up to some translation in $x$ (still using the same notation for the translated subsequence) that converges to  $\left (z^{0} , w^{0} \right )$ in $ H^1 (\mathbb{R}) \times H^1 (\mathbb{R})$.}
\smallskip

\noindent The proof of Claim I is given in Appendix \ref{B}. Thus, from this claim, \eqref{ccccc} implies that $z^0 = w^0$ and $$\left (z^{\ep_j} , w^{\ep_j} \right )\rightarrow \left (w^{0} , w^{0} \right ) \ \text{in} \ \ H^1 (\mathbb{R}) \times H^1 (\mathbb{R}).$$ Moreover, $$K^0 (w^0 ) = \lim_{j \rightarrow \infty} G\left (z^{\ep_j} , w^{\ep_j} \right ) = -1$$ and
\begin{align*}
{\mathcal{J}}^0 \leq& J^0( w^0 ) \\=&  \lim_{j \rightarrow \infty} \int_{\mathbb R}\Bigg ( \left (w^{\ep_j}\right )^2  + |c|\left(\left (z^{\ep_j} \right ) '  -  \frac{b\omega(\ep_j)}{|c|} \left (w^{\ep_j}\right ) '\right)^2 \\
& +\left(\frac{ac-b^2 \omega^2(\epsilon_j)}{|c|}\right) \left (\left ( w^{\ep_j}\right )'\right )^2\Bigg)\,dy\\
\leq &\lim_{j \rightarrow\infty } I^{\ep_j} \left (z^{\ep_j} , w^{\ep_j} \right ) = \liminf_{\ep \rightarrow 0^+} {\mathcal{I}}^\epsilon \leq \limsup_{\ep \rightarrow 0^+} {\mathcal{I}}^\epsilon \leq {\mathcal{J}}^0\, ,
\end{align*}
where \eqref{RR} has been used.
\color{black}
Hence, we have that
$$
\lim _{\epsilon \rightarrow 0^{+}} \mathcal{I}^\epsilon=\mathcal{J}^0\, .
$$
Finally, for any $w\in H^1(\mathbb{R})$ satisfying $K^\epsilon(w)=-1$, it is known that $\mathcal{I}^\epsilon \leq J^\epsilon(w)$ due to the fact that  $\mathcal{I}^\epsilon \leq \mathcal{J}^\epsilon$ and \eqref{JJe}. Thus, we get that \color{black}
$$
\mathcal{J}^0=\lim _{\epsilon \rightarrow 0^{+}} \mathcal{I}^\epsilon \leq \liminf _{\epsilon \rightarrow 0^{+}} \mathcal{J}^\epsilon.
$$
Again, using \eqref{RR}, we conclude that
$$
\lim _{\epsilon \rightarrow 0^{+}} \mathcal{J}^\epsilon=\mathcal{J}^0,
$$
showing \eqref{lim_2}, and the lemma is achieved.
\end{proof}

Before we go further, we characterize the solitary-wave solutions for the generalized KdV equation. In the one-dimensional case, the following result is a consequence of the results shown in \cite{DBJCb}, where the authors characterize the solitary-wave solutions for the (KP-I) model in the $d$-dimensional case.
\begin{thm}\label{tkdv}
Let $\{w_m\}_{m\geq0}$ be a minimizing sequence for $\mathcal J^0$ given by Lemma \ref{limite}. Then,  there exists a sequence of points $(y_n)_m \subset \R$ and a subsequence, which will be denoted by the same index, and a nonzero $w_0 \in H^1(\R)$ such that $J^0(w_0) = \mathcal J^0,$ and $$w_m(\cdot + y_m) \to w_0 \in H^1(\R).$$ Moreover, $w_0$ is a solution to the equation
\begin{equation}\label{kdv-t}
w_0+\left(\frac13-\sigma\right) w_{0xx}+\frac{2}{(p+1)}\mathcal J^0 w_0^{p+1}=0\, .
\end{equation}
  Therefore, $\tilde w=\left(\frac2{p+1} \mathcal J^0 \right)^{\frac1{p}}w_0$ is a nontrivial solitary wave solution for the generalized KdV equation \eqref{kdv}, i.e.,
\begin{align*}
  \tilde w+\left(\frac13-\sigma\right) \tilde w_{xx}+ \tilde w^{p+1}=0\, .
\end{align*}
\end{thm}

We are in a position to prove the main result of this section 
which is a consequence of Theorem \ref{tkdv}.  We will see that the translated subsequence of the renormalized sequence $\{\left(z^{\ep_j},w^{\ep_j}\right)\}_j$ converges to a function $w_0$ that satisfies the system \eqref{trav-eqs}. Thus, $w_0$ is a solution of the generalized  KdV equation.

\begin{thm}\label{conver}
For any sequence $\ep_j \to 0^+$ there is a translated subsequence \newline $\{\left(z^{\ep_j},w^{\ep_j}\right)\}_j$,  and a nontrivial $(z_0,w_0)\in X$ such that
\begin{equation}\label{rr}
(z^{\ep_j},w^{\ep_j}) \to (z_0,w_0) \ \ \mbox{in $ X$},\quad \mbox{and}\quad z^{\ep_j}-w^{\ep_j} \to 0,  \mbox{ as $j\to \infty$}.
\end{equation}
Moreover, $(z_0,w_0)$ is a nontrivial solution of the system
\begin{equation*}
\begin{cases}
z_0 =  w_0 \\
w_0+ \left(\frac13-\sigma\right)  w_0'' + \frac{2}{(p+1)} {\mathcal J}^0 w_0^{p+1}=0.
\end{cases}
\end{equation*}
In other words, $z_0 = w_0 \in  H^1(\R)$, with $w_0$ being a traveling wave for the generalized KdV equation \eqref{kdv-t}.
\end{thm}
\begin{proof}
Let $\{\ep_j\}_j$ be a sequence of positive number such that $\ep_j \to 0^{+}$, when $j\to\infty$. From Lemma \ref{limite}, we note that $\left(-K^0\left(w^{\ep_j}\right)^{-\frac1{p+2}}w^{\ep_j}\right)_j $ is a minimizing sequence for $\mathcal J^0$ and also that $K^0\left(w^{\ep_j}\right) \to -1.$  Thanks to the previous convergence and Theorem \ref{tkdv}, there exist a translated sequence of $\{ w^{\ep_j}\}_j$ and a nonzero function $w_0\in H^1(\R)$ such that $w^{\ep_j} \to w_0$ in $H^1(\R)$ and $w_0$ is a solution of \eqref{kdv-t}.  Additionally, from \eqref{I-ep} and the uniform boundedness of $\{(z^{\ep_j}, w^{\ep_j})\}_j$ in $H^1(\mathbb{R})$, it is obtained that
$$
\left\|z^{\epsilon_j}-\omega(\epsilon_j) w^{\epsilon_j}\right\|_{L^2\left(\mathbb{R}\right)}=O(\epsilon^{\frac{1}{p+2}}),
$$
which implies that there exists a nontrivial  $z_0\in L^2(\R)$ such that $z^{\ep_j} \to z_0$ in $L^2(\R)$ and $z_0 =w_0$. Moreover, since $(z^{\ep_j}, w^{\ep_j})$ is a minimizer of $I^{{\ep_j}} ( z , w )$ with $G(z, w ) = -1$ and $\lim _{j \rightarrow \infty} I^{{\ep_j}} ( z^{\ep_j}, w^{\ep_j} ) = {\mathcal{J}}_0$, a concentration-compactness argument shows that $z_0 \in H^1(\R)$ and $z^{\ep_j} \to z_0$ in $H^1(\R)$, which gives \eqref{rr}.

Now, considering a test function  $\xi\in C^{\infty}(\mathbb{R})$ and using the system \eqref{trav-ep}, we get
\begin{equation*}
\left\langle\ep^{-\frac{2}{p+2}}_{j}(w^{\ep_j}-\omega(\ep_j) z^{\ep_j})+\omega(\ep_j) b(z^{\ep_j})''+a(w^{\ep_j})'',\xi \right\rangle=-\left\langle\frac{2}{p+2}\mathcal I^{\ep_{j}} z^{\ep_j}(w^{\ep_j})^{p},\xi\right\rangle
\end{equation*}
and
\begin{equation*}
\begin{split}
&\left\langle\ep^{-\frac{2}{p+2}}_{j}(z^{\ep_j}-\omega(\ep_j) w^{\ep_j})+\omega(\ep_{j}) b(w^{\ep_j})''+c(z^{\ep_j})'',\xi \right\rangle=\\
& \qquad\qquad\qquad\qquad\qquad\qquad\qquad\qquad- \left\langle\frac{2}{(p+1)(p+2)} \mathcal I^{\ep_{j}} (z^{\ep_j})^{p+1},\xi\right\rangle.
\end{split}
\end{equation*}
Adding both equations in the previous system, we find that
\begin{equation*}
\begin{split}
\left\langle\ep^{-\frac{2}{p+2}}_{j}(1-\omega(\ep_j))(w^{\ep_j} +z^{\ep_j}) + (b\omega(\ep_j)+a)(w^{\ep_j})''+(b\omega(\ep_j)+c)(z^{\ep_j})'',\xi\right\rangle\\
= -\left\langle\frac{2}{(p+1)(p+2)} \mathcal I^{\ep_j}(z^{\ep_j})^{p+1}+\frac{2}{p+2} \mathcal I^{\ep_j}z^{\ep_{j}}(w^{\ep_j})^{p},\xi\right\rangle.
\end{split}
\end{equation*}
Note that using the first part of the proof yields $z_0 =w_0$,  when $\ep_{j}\to 0^{+}$. Moreover, since $1-\omega^2(\ep_j)=\ep_j^{\frac{2}{p+1}}$ gives that $$\lim_{\ep_{j}\to0^{+}}\ep^{-\frac{2}{p+2}}_{j}(1-\omega(\ep_j))=\frac{1}{2}\, ,$$
thus,
\begin{equation*}
\begin{split}
\lim _{j\to\infty}\left\langle\ep^{-\frac{2}{p+2}}_{j}(1-\omega(\ep_j))(w^{\ep_j}+z^{\ep_j}) + (b\omega(\ep_j)+a)(w^{\ep_j})''+(b\omega(\ep_j)+c)(z^{\ep_j})'',\xi\right\rangle=\\
\left\langle w_{0} +(2b+a+c)w''_{0},\xi\right\rangle.
\end{split}
\end{equation*}
Thanks to the fact that $\mathcal J^{\ep_j} \to \mathcal J^0$, we have
\begin{equation*}
\begin{split}
-\lim _{j\to\infty}\left\langle\frac{2}{(p+1)(p+2)} \mathcal I^{\ep_j}(z^{\ep_j})^{p+1}+\frac{2}{p+2} \mathcal I^{\ep_j}z^{\ep_{j}}(w^{\ep_j})^{p},\xi\right\rangle\\=-\left\langle \frac{2}{(p+1)}\mathcal I^{0}w_{0}^{p+1},\xi\right\rangle.
\end{split}
\end{equation*}
Finally,  putting previous equalities together gives the existence of the non-trivial solution $w_0=z_0$ to the equation
\begin{equation}\label{limit}
w_0 + \left (\frac{1}{3}-\sigma\right )w_0'' +\frac{2}{(p+1)} \mathcal I^{0} w_0^{p+1}=0,
\end{equation}
as desired once we have that $a+c+ 2b=\frac13-\sigma$, showing the result.
\end{proof}

\begin{rem}
We note that $\left(z^{\varepsilon}, w^{\varepsilon}\right)$ is a solution of \eqref{trav-ep}, which can be rewritten as
\begin{equation}
\begin{split}
\epsilon^{-2/(p+1 ) } ( w^\epsilon - z^\epsilon ) &+
\epsilon^{-2/(p+1 ) } ( 1- \omega) z^\epsilon
+ \omega b  z^\epsilon_{yy} \\  &+a   w^\epsilon_{yy} + \frac{2}
{p+2 } {\mathcal I}^\epsilon  \left ( w^\epsilon \right ) ^{p} z^\epsilon= 0 \label{zzz_1}
\end{split}
\end{equation}
and
\begin{equation}
\begin{split}
\epsilon^{-2/(p+1 ) } ( z^\epsilon -  w^\epsilon ) &+
\epsilon^{-2/(p+1 ) } ( 1 - \omega)  w^\epsilon+ \omega b  w^\epsilon_{yy}
\\&  +c z^\epsilon_{yy} + \frac{2}
{(p+1) (p+2) } {\mathcal I}^\epsilon  \left ( w^\epsilon \right) ^{p+1} = 0 \, .\label{zzz_2}
\end{split}
\end{equation}
By the fact that  $ \omega^2  = 1 - \epsilon^{2/(p+1 ) }$ with Lemma \ref{limite} and Theorem \ref{conver}, we have
$$
w^\epsilon\rightarrow w_0 \, ,\quad  z^\epsilon\rightarrow w_0\, ,\quad {\mathcal I}^\epsilon\rightarrow {\mathcal J}^0,\qquad \mbox{as}\quad \epsilon \rightarrow 0\, .
$$
However, since $\epsilon^{-2/(p+1 ) }\rightarrow +\infty$ as $\epsilon \rightarrow 0$, the term
$\epsilon^{-2/(p+1 ) } ( w^\epsilon - z^\epsilon ) $ in \eqref{zzz_1} and \eqref{zzz_2} may not approach to zero. To obtain the limit of $\epsilon^{-2/(p+1 ) } ( w^\epsilon - z^\epsilon ) $,
one needs to derive the next orders of $w^\epsilon, z^\epsilon$, i.e.,
\begin{align}
w^\epsilon = w_0 + w_1 \epsilon^{2/(p+1 ) }\, \qquad \text{and}\ \quad z^\epsilon =  w_0 + z_1 \epsilon^{2/(p+1 ) }\,  ,\label{zzz_3}
\end{align}
where, in general, $w_1 - z_1 $ may not tend to zero as $\epsilon \rightarrow 0$. In fact, the limits of $w_1 $ and $z_1$ as $\epsilon \rightarrow 0$ can be found from \eqref{zzz_1} and \eqref{zzz_2} as well.
By adding \eqref{zzz_1} and \eqref{zzz_2} and using \eqref{zzz_3} and the equation for $w_0$, we can derive one equation for $w_1 $ and $z_1$ in terms of $w_0$. Another equation is directly from \eqref{zzz_1} using \eqref{zzz_3}. Therefore, those
two equations yield the limits of $w_1 $ and $z_1$ as $\epsilon \rightarrow 0$. In this way, we can derive the asymptotic forms of $w^\epsilon$ and $z^\epsilon$ up to any order of $\epsilon$ as $\epsilon \rightarrow 0$.
\end{rem}
\begin{rem}Another interesting fact is that the profiles $\left(\psi_\omega, v_\omega\right)$ can be small or ``large'' for $\omega$ admissible. However, the profile is small when $\omega $ is near $1^-$, which  can be seen from the fact that  $I_{2,\omega}(\psi, v) < 0$ for $\omega$ close to $1^-$ and 
$$\|(\psi_\omega, v_\omega)\|{H^1 \times H^1} \sim I_1(\psi_\omega, v_\omega) \leq I_\omega(\psi_\omega, v_\omega) = \mathcal{I}_\omega = \epsilon^{\frac{p+4}{(p+1)(p+2)}} \mathcal{I}^\epsilon \to 0, \quad \text{as } \omega \to 1^-.$$
Therefore, for $\omega$ sufficiently close to $1^-$, $\|(\psi_\omega, v_\omega)\|_{H^1 \times H^1} $ is small.
\end{rem}


\section{GSS approach}\label{sec4}
Recall that solitary waves are characterized as critical points of a function defined on a suitable space. For our generalized $abcb$-Boussinesq system \eqref{1bbl},  remember that the functional $J_{\omega}: X^*\to X$ is given by \eqref{JJ_w}, where $X^*$ is the dual space of $X$.  Hereafter, a solitary-wave solution (or a traveling-wave solution of finite energy) minimizes the action functional $J_{\omega}$ under some constraints.

It is noted here that we are only interested in the stability of solutions in the set $\tilde {\mathcal {G}}_\omega$ when $\omega $ is near $1^-$, where $\tilde {\mathcal {G}}_\omega$ is the set of the solutions $(\tilde \psi_\omega, \tilde v_\omega)$ of \eqref{trav-eqs} that correspond to the minimizers $(\psi_\omega,  v_\omega)$ of \eqref{mp}.
Now, given any $(\tilde \psi_\omega, \tilde v_\omega)$ of \eqref{trav-eqs}  in $\tilde {\mathcal {G}}_\omega$, define 
\begin{equation}\label{d_w1}
d(\omega) = \frac{p}{2(p+2)}I_\omega(\tilde \psi_\omega, \tilde v_\omega),
\end{equation}
which is well defined due to the fact that using Section \ref{sec2}, if $(\psi_{\omega},v_{\omega})$ is a minimizer of the problem given by \eqref{mp} and ${\mathcal{G}}_{\omega}$ is the set of all such $(\psi_{\omega},v_{\omega})$, then 
 Theorem \ref{wsol} implies 
\begin{equation}
(\tilde \psi_\omega ,\tilde v_\omega )=\left(\frac{2}{p+2}\mathcal{I}_{\omega}\right)^{\frac{1}{p}}(\psi_{\omega},v_{\omega}), \label{zzzzz}
\end{equation}
is solution of \eqref{trav-eqs} and Remark \ref{rems} yields that
 \begin{equation}\label{d_w}
 \begin{split}
d(\omega) 
= & \frac{p}{2(p+2)}I_\omega\left(\left(\frac{2}{p+2}\mathcal{I}_{\omega}\right)^{\frac{1}{p}}(\psi_{\omega},v_{\omega})\right)\\
=& \frac{p}{2(p+2)}\left(\frac{2}{p+2}\mathcal{I}_{\omega}\right)^{\frac{2}{p}}I_\omega(\psi_{\omega},v_{\omega}) \\
= & \frac{p}{2(p+2)}\left(\frac{2}{p+2}\right)^{\frac{2}{p}}\left ( I_\omega(\psi_{\omega},v_{\omega})\right ) ^{\frac{p+2}p},
\end{split}
\end{equation}
which is uniquely defined  for any $(\psi_{\omega},v_{\omega}) \in \mathcal{G}_\omega$, i.e., 
independent of $({\tilde \psi}_{\omega}, {\tilde v}_{\omega}) \in \tilde { \mathcal{G}}_\omega$. 

We remark that, as proposed by Shatah \cite{Shatah}, for the study of the stability of the standing waves of nonlinear Klein-Gordon equations, and Grillakis, Shatah, and Strauss \cite[Theorems 2 and 3]{GSS}, considering an abstract Hamiltonian system, the analysis of the stability of solution sets depends upon some properties of the scalar function given by
 \begin{equation}\label{d_w1rr}
 d(\omega)=\mathcal H(\psi,v)+\omega\mathcal{Q}(\psi,v)=\mathcal{F_{\omega}}(\psi,v)=\frac{1}{2}J_{\omega}(\psi,v)\, ,
\end{equation}
where $(\psi,v)$ is in a set of solutions of \eqref{trav-eqs} that makes such $d(\omega)$ well defined.  For the solution set $\tilde {\mathcal{G}}_\omega$ studied here, Remark \ref{rems} implies that the definitions of $d(\omega)$ in \eqref{d_w1} and \eqref{d_w1rr} are same. 
\subsection{Properties of the scalar function} This subsection is devoted to presenting properties of the scalar function $d(\omega)$\footnote{It is important to point out again that this strategy was originally introduced in \cite{Shatah}, and after that, extended for an abstract framework in \cite{GSS}.}  when  $\omega$ is near $1^-$. From now on, we will use the notation of the previous section, the characterization of $d(\omega)$ given by \eqref{d_w}, and we take into account the relation \eqref{relation_c}. Since $(\psi_\omega, v_\omega) \in \mathcal{G}_\omega$ is uniformly bounded for $\omega $ in a closed interval  of $(0, 1)$ using \eqref{mp}, by the definitions of $I_\omega$ in \eqref{I_w} and  $\mathcal{I}_\omega$ in \eqref{mp}, it is straightforward to show that $\mathcal{I}_\omega$ in \eqref{mp} is continuous, which implies that $d(\omega)$ is continuous for $\omega \in (0, 1)$. Now, the following result gives us a relation for $0<{\omega}_1<{\omega}_2<1$ in terms of $d$.
\begin{lem}\label{estd}
For $0<{\omega}_1<{\omega}_2<1$ and $(\psi_{{\omega}_i},v_{{\omega}_i})\in\mathcal{G}_{{\omega}_i}$,  $i=1,2,$ it follows that
\begin{equation}\label{dwi}
\begin{split}
d({\omega}_1)\leq &d({\omega}_2)-\frac{1}{2}\left(\frac{2}{p+2}\right)^{\frac{2}{p}}(I_{\omega_{2}}(\psi_{\omega_{2}},v_{\omega_{2}}))^{\frac{2}{p}}\left(\frac{w_{2}-w_{1}}{w_{2}}\right)I_{2,\omega_{2}}(\psi_{\omega_{2}},v_{\omega_{2}})\\&+O(({\omega}_2-{\omega}_1)^{2})
\end{split}
\end{equation}
and
\begin{equation}\label{dwii}
\begin{split}
d({\omega}_2)\leq& d({\omega}_1)+\frac{1}{2}\left(\frac{2}{p+2}\right)^{\frac{2}{p}}(I_{\omega_{1}}(\psi_{\omega_{1}},v_{\omega_{1}}))^{\frac{2}{p}}\left(\frac{w_{2}-w_{1}}{w_{1}}\right)I_{2,\omega_{1}}(\psi_{\omega_{1}},v_{\omega_{1}})\\&+O(({\omega}_2-{\omega}_1)^{2}).
\end{split}
\end{equation}
\end{lem}
\begin{proof}
Due to the definition of $d(\omega)$, given by \eqref{d_w}, we have that
\begin{equation*}
\begin{split}
d({\omega}_1)=&\frac{p}{2(p+2)}\left(\frac{2}{p+2}\right)^{\frac{2}{p}}(I_{\omega_{1}}(\psi_{\omega_{1}},v_{\omega_{1}}))^{\frac{p+2}{p}}\\
\leq &\frac{p}{2(p+2)}\left(\frac{2}{p+2}\right)^{\frac{2}{p}}(I_{\omega_{1}}(\psi_{\omega_{2}},v_{\omega_{2}}))^{\frac{p+2}{p}}\\
=&\frac{p}{2(p+2)}\left(\frac{2}{p+2}\right)^{\frac{2}{p}}\Big (I_{1}(\psi_{\omega_{2}},v_{\omega_{2}})+I_{2,\omega_{2}}(\psi_{\omega_{2}},v_{\omega_{2}})\\
& \qquad\qquad\qquad\qquad -I_{2,\omega_{2}}(\psi_{\omega_{2}},v_{\omega_{2}})+I_{2,\omega_{1}}(\psi_{\omega_{2}},v_{\omega_{2}})\Big)^{\frac{p+2}{p}}\\
=&\frac{p}{2(p+2)}\left(\frac{2}{p+2}\right)^{\frac{2}{p}}\left(I_{{\omega}_2}(\psi_{{\omega}_2},v_{{\omega}_2})+\frac{{\omega}_1-{\omega}_2}{{\omega}_2}I_{2,{\omega}_2}(\psi_{{\omega}_2},v_{{\omega}_2})\right)^{\frac{p+2}{p}},
\end{split}
\end{equation*}
thanks to \eqref{I_w} and to the fact that $I_{2,\omega_{1}}(\psi_{\omega_{2}},v_{\omega_{2}})=\frac{\omega_{1}}{\omega_{2}}I_{2,\omega_{2}}(\psi_{\omega_{2}},v_{\omega_{2}})$. Thus, using Taylor's series of a power function with power $(p+2)/p$ around $\omega_1 - \omega_2$ near zero in the previous inequality, we find
\begin{equation*}
\begin{split}
d({\omega}_1)\leq&\frac{p}{2(p+2)}\left(\frac{2}{p+2}\right)^{\frac{2}{p}}(I_{\omega_{2}}(\psi_{\omega_{2}},v_{\omega_{2}}))^{\frac{p+2}{p}}\\&-\frac{p}{2(p+2)}\left(\frac{2}{p+2}\right)^{\frac{2}{p}}\frac{p+2}{p}(I_{\omega_{2}}(\psi_{\omega_{2}},v_{\omega_{2}}))^{\frac{2}{p}}\left(\frac{w_{2}-w_{1}}{w_{2}}\right)I_{2,\omega_{2}}(\psi_{\omega_{2}},v_{\omega_{2}})\\&+O((w_{2}-w_{1})^{2})\\
=&d(\omega_{2})-\frac{1}{2}\left(\frac{2}{p+2}\right)^{\frac{2}{p}}(I_{\omega_{2}}(\psi_{\omega_{2}},v_{\omega_{2}}))^{\frac{2}{p}}\left(\frac{w_{2}-w_{1}}{w_{2}}\right)I_{2,\omega_{2}}(\psi_{\omega_{2}},v_{\omega_{2}})\\&+O((w_{2}-w_{1})^{2})
\end{split}
\end{equation*}
and the inequality \eqref{dwi} is verified. The proof of \eqref{dwii} is analogous and will be omitted.
\end{proof}

We are now in a position to characterize $d'(\omega)$. However, from Lemma \ref{estd}, we need to evaluate $I_{2, \omega} (\psi_\omega, v_\omega)$ as $\omega \rightarrow \omega_0$. In general, such a limit may not be well-defined and may depend on the choice of $(\psi_{\omega_0}, v_{\omega_0}) \in \mathcal{G}_{\omega_0}$. Moreover, in the worst case, this $(\psi_{\omega_0}, v_{\omega_0}) $ may be a singleton in $ \mathcal{G}_{\omega_0}$ and there is no $(\psi_{\omega}, v_{\omega}) \in \mathcal{G}_{\omega}$ near $ (\psi_{\omega_0}, v_{\omega_0})$ when $\omega$ is sufficiently close to $\omega_0$. Thus, we need to use the local minimizers defined in Remark \ref{yyyyyy} and prove the following lemma.
\begin{lem}\label{prepare}
For any given $\omega_0 \in (0,1) $ close to $1$ and a $(\psi_{\omega_0}, v_{\omega_0}) \in \mathcal{G}_{\omega_0}$, there is a $\delta_0  >0 $ small enough such that $\Delta_0 = (\omega_0 - \delta_0, \omega_0 + \delta_0) \in (0, 1)$ and for every  $\omega \in \Delta_0$, there is a unique 
$\{ (\psi_{\omega}, v_{\omega}) \} \in \mathcal{G}_{\omega}$ (or at least a local minimizer) and $(\psi_{\omega}, v_{\omega})\rightarrow (\psi_{\omega_0}, v_{\omega_0})$ in $H^1(\mathbb{R})\times H^1(\mathbb{R})$ as $\omega\rightarrow \omega_0$.
\end{lem}
The proof of this lemma will be given in Appendix \ref{C}. In the following, whenever $d'(\omega)$ and $d''(\omega)$ are calculated, we always mean that the calculations are based upon a choice of $(\psi_{\omega}, v_{\omega}) \in \mathcal{G}_{\omega}$ where the limits in the derivatives may be taken using local or global minimizers. Therefore, in Lemma \ref{estd}, when $\omega_1$ is near $\omega_2$ and one of $(\psi_{\omega_i }, v_{\omega_i}), i=1,2,$ is a local minimizer, we may use $\hat d(\omega)$ for the minimizer without global minimizers nearby in stead of $d(\omega)$,  making the notations different. Hence, in the following, the definition of $d(\omega)$ may be based on a fixed minimizer corresponding to $\omega$.
\begin{lem}\label{deri_a}
For $(\psi_{\omega},v_{\omega})\in\mathcal{G}_{\omega}$, with $0<\tilde \omega_{0}<\omega<1$, and \eqref{relation_c} being satisfied, it follows that
\begin{align}\label{deri}
d\, '({\omega})=\frac{1}{2}\left(\frac{2}{p+2}\right)^{\frac{2}{p}}\frac{I_{2,{\omega}}(\psi_{\omega},v_{\omega})}{\omega}(I_{\omega}(\psi_{\omega},v_{\omega}))^{\frac{2}{p}}= \mathcal Q(\tilde \psi_{\omega},\tilde v_{\omega})\, ,
\end{align}
or $\hat d\, ' (\omega)$. Additionally, we have that $d\, '(\omega)<0$ (or $\hat d\, '(\omega)<0$), when $\omega$ is near to $1^{-}$.
\end{lem}
\begin{proof}
From the previous lemma, we have that
\begin{align*}
\frac{d({\omega}_1)- d({\omega}_2)}{\omega_1-\omega_2}& \geq \frac{1}{2}\left(\frac{2}{p+2}\right)^{\frac{2}{p}}(I_{\omega_{2}}(\psi_{\omega_{2}},
v_{\omega_{2}}))^{\frac{2}{p}}\left(\frac{1}{\omega _{2}}\right)I_{2,\omega_{2}}(\psi_{\omega_{2}},v_{\omega_{2}})+
O(({\omega}_2-{\omega}_1))\end{align*}
and
\begin{align*}
\frac{d({\omega}_2)- d({\omega}_1)}{\omega_2-\omega_1}&\leq \frac{1}{2}\left(\frac{2}{p+2}\right)^{\frac{2}{p}}(I_{\omega_{1}}(\psi_{\omega_{1}},
v_{\omega_{1}}))^{\frac{2}{p}}\left(\frac{1}{\omega_{1}}\right)I_{2,\omega_{1}}(\psi_{\omega_{1}},v_{\omega_{1}})
+O(({\omega}_2-{\omega}_1)).
\end{align*}
Taking limit as $\omega_1 \to \omega_2$ in the first inequality and $\omega_2 \to \omega_1$ in the second inequality, we get that
\begin{align*}
(d')^{-}(\omega_2)& \geq \frac{1}{2}\left(\frac{2}{p+2}\right)^{\frac{2}{p}}(I_{\omega_{2}}(\psi_{\omega_{2}},
v_{\omega_{2}}))^{\frac{2}{p}}\left(\frac{1}{\omega_{2}}\right)I_{2,\omega_{2}}(\psi_{\omega_{2}},v_{\omega_{2}}),
\end{align*}
and
\begin{align*}
(d')^{+}(\omega_1)&\leq \frac{1}{2}\left(\frac{2}{p+2}\right)^{\frac{2}{p}}(I_{\omega_{1}}(\psi_{\omega_{1}},
v_{\omega_{1}}))^{\frac{2}{p}}\left(\frac{1}{\omega_{1}}\right)I_{2,\omega_{1}}(\psi_{\omega_{1}},v_{\omega_{1}}),
\end{align*}
which implies that for any $\omega\ne0$,
\[
d'(\omega)= \frac{1}{2}\left(\frac{2}{p+2}\right)^{\frac{2}{p}}(I_{\omega}(\psi_{\omega},
v_{\omega}))^{\frac{2}{p}}\left(\frac{1}{\omega}\right)I_{2,\omega}(\psi_{\omega},v_{\omega}).
\]
Next equality is derived from 
\begin{align*}
d(\omega) = & \mathcal{H} (\tilde \psi_{\omega},\tilde v_{\omega}) + \omega Q (\tilde \psi_{\omega},\tilde v_{\omega}) = \frac p {2(p+2)} I_\omega (\tilde \psi_{\omega},\tilde v_{\omega}) \\
= & \frac12\left ( I_\omega (\tilde \psi_{\omega},\tilde v_{\omega})  - I_{2,\omega} (\tilde \psi_{\omega},\tilde v_{\omega}) + G (\tilde \psi_{\omega},\tilde v_{\omega}) \right ) + \omega Q (\tilde \psi_{\omega},\tilde v_{\omega}) \, ,
\end{align*}
which, by \eqref{critical3}, implies $Q (\tilde \psi_{\omega},\tilde v_{\omega}) = \frac1{2\omega} I_{2,\omega} (\tilde \psi_{\omega},\tilde v_{\omega})$ and the equality using \eqref{zzzzz}.

Finally, by the quantity $I^{2,\epsilon}(z,w)$ defined in \eqref{i2e}, that is, from the scaling \eqref{defuz} and $I_{2, \omega}$ (see \eqref{NEW-R}), we ensures that
\begin{equation}\label{rrrr}
\begin{split}
I^{2,\epsilon}(z^\epsilon,w^\epsilon)
=&-2\sqrt{1-\epsilon^{\frac{2}{p+1}}}\int_{\mathbb{R}}\Bigl(\epsilon^{-\frac{2}{p+1}}z^\epsilon w^\epsilon+b(\partial_yz^\epsilon)(\partial_{y}w^\epsilon)\Bigr)dy,
\end{split}
\end{equation}
since $\omega^2(\ep)=1-\ep^{\frac{2}{p+1}}$. Passing the limit when $\epsilon\rightarrow0^+$ in \eqref{rrrr} and thanks to the Theorem \ref{conver},  we obtain
$$
\lim_{\epsilon\rightarrow0^+} \epsilon^{\frac{2}{p+1}}
I^{2,\epsilon}(z^\epsilon,w^\epsilon)= -2 \int_{\mathbb{R}}w_0^2dy <0.
$$
 This means that  $I^{2,\epsilon}(z^\epsilon,w^\epsilon)<0$ for $\epsilon$ near $0^+$, which implies
$
I_{2,{\omega}}(\psi_{\omega},v_{\omega})<0,
$
and, due to the expression \eqref{deri}, we find $d'(\omega)<0$.
\end{proof}

\begin{rem} From the proof of Lemma \ref{deri_a}, we see that $d'(\omega) $ may depend on the minimizer $(\psi_\omega, v_\omega)$ of \eqref{mp}. In fact, due to the nature of global minimizers in \eqref{mp}, $d' (\omega) $ must be equal for all $(\psi_\omega, v_\omega)\in \mathcal{G}_\omega$ (here, even for $\hat d'(\omega)$ obtained from local minimizers if no global minimizers are nearby), since, if otherwise, the curves for $d(\omega)$ near two different minimizers $(\psi_\omega, v_\omega)$ will cross each other which contradicts to the global minimizer of \eqref{mp}. Moreover, due to the same reason, $\hat d'' (\omega ) \geq  d'' (\omega)$ (see Lemma \ref{lem47}), where $d''(\omega )$ is obtained from the minimizer that has a global minimizer branch at $\omega$.
\end{rem}

With the previous lemma and remark in hand, let us give a relation for $d''(\omega)$ when $\omega$ is near $1^{-}$, which will ensure the convexity of $d$.

\begin{lem}\label{lem47}
Suppose that \eqref{relation_c} holds. Then, for $0<{\omega}<1$ near $1^{-}$,  it follows that
\begin{equation}\label{deri2}
\begin{split}
d''({\omega})=&\frac{1}{p}\left(\frac{2}{p+2}\right)^{\frac{2}{p}}\left(\frac{I_{2,{\omega}}(\psi_{\omega},v_{\omega})}{\omega}\right)^{2}(I_{\omega}(\psi_{\omega},v_{\omega}))^{\frac{2-p}{p}}\\&+\frac{1}{2}\left(\frac{2}{p+2}\right)^{\frac{2}{p}}(I_{\omega}(\psi_{\omega},v_{\omega}))^{\frac{2}{p}}\frac{d}{d\omega}\left(\frac{I_{2,{\omega}}(\psi_{\omega},v_{\omega})}{\omega}\right).
\end{split}
\end{equation}
Moreover, when $0<{\omega}<1$ is near $1^{-}$, $d''(\omega)>0$  if $ 1\leq p < p_0 $ and  $d''(\omega)< 0$ if $ p > p_0 $, where $p_0> 4$ is the unique positive root of
$$
\left ( \frac{p+2}{p+1}\right )^{\frac{2}{p}} - \frac{p^2}{2(p+4)} = 0 \, .
$$
A similar conclusion holds for $\hat d\, '' (\omega)$.
\end{lem}
\begin{proof}
Differentiating the equation \eqref{deri} in terms of $\omega$ and taking into account that $d(\omega)$ is given by \eqref{d_w}, straightforward calculations show that the relation \eqref{deri2} holds.  Now, since the first term of the right-hand side of \eqref{deri2} is explicit and positive, thanks to the fact that $I_{\omega}(\psi_{\omega},v_{\omega})>0$, we only need to  find that
\begin{align*}
\frac{d}{d\omega}\left(\frac{I_{2,{\omega}}(\psi_{\omega},v_{\omega})}{\omega}\right) & = \frac{d\ep}{d\omega} \frac{d}{d\epsilon}\left( \ep^{\frac{p+4}{(p+1)(p+2)}}\frac{I^{2,{\epsilon}}(z(y),w(y))}{\omega(\epsilon)}\right)\\
& = - \omega (p+1) \ep^{\frac{p-1}{p+1}}\frac{d}{d\epsilon}\left( \ep^{\frac{p+4}{(p+1)(p+2)}}\frac{I^{2,{\epsilon}}(z(y),w(y))}{\omega(\epsilon)}\right),
\end{align*}
due to the relation \eqref{i2e} and $\omega ^2 = 1 - \ep^{\frac 2 {p+1}}$.

Observe that using the notation $(z(y),w(y))=(z(y,\epsilon),w(y,\epsilon))$, we obtain that
\begin{equation*}
\begin{split}
\frac{d}{d\epsilon}& \left( \ep^{\frac{p+4}{(p+1)(p+2)}}\frac{I^{2,{\epsilon}}(z(y),w(y))}{\omega(\epsilon)}\right)\\
& = \frac{d}{d\epsilon} \left( \ep^{\frac{p+4}{(p+1)(p+2)}}
(-2) \int_{\mathbb{R}} \left ( \ep^{-\frac2{p+1} } z(y , \ep ) w ( y, \ep ) + bz_y(y , \ep ) w_y ( y, \ep )\right ) dy \right)\\
& = -2 \frac{d}{d\epsilon} \left( \ep^{- \frac{p}{(p+1)(p+2)}}
\int_{\mathbb{R}} \left (  z(y , \ep ) w ( y, \ep ) + \ep^{\frac2{p+1} } bz_y(y , \ep ) w_y ( y, \ep )\right ) dy \right)\\
& = \frac{2p}{(p+1)(p+2)}  \left( \ep^{- \frac{p}{(p+1)(p+2)}-1 }
\int_{\mathbb{R}} \left (  z(y , \ep ) w ( y, \ep ) + \ep^{\frac2{p+1} } bz_y(y , \ep ) w_y ( y, \ep )\right ) dy \right)\\
& \qquad  -2 \left( \ep^{- \frac{p}{(p+1)(p+2)}}\frac{d}{d\epsilon}
\int_{\mathbb{R}} \left (  z(y , \ep ) w ( y, \ep ) + \ep^{\frac2{p+1} } bz_y(y , \ep ) w_y ( y, \ep )\right ) dy \right)\\
& = \frac{2p}{(p+1)(p+2)}  \left( \ep^{- \frac{p}{(p+1)(p+2)}-1 }
\int_{\mathbb{R}} \left (  z(y , \ep ) w ( y, \ep ) + \ep^{\frac2{p+1} } bz_y(y , \ep ) w_y ( y, \ep )\right ) dy \right)\\
& \qquad  + \ep^{- \frac{p}{(p+1)(p+2)}} R( z, w, z_\ep, w_\ep),
\end{split}
\end{equation*}
where the subscripts $\epsilon$ and $y$ mean the derivatives with respect these variables and $R( z, w, z_\ep, w_\ep) $ is linear in terms of either $z_\ep$ or $w_\ep$. Taking the limit when $\epsilon$ goes to $0^{+}$ in the previous identity,
from Theorem \ref{conver} and Lemma \ref{derivative}, we can see that
\begin{equation*}
\begin{split}
\lim_{\ep \to 0^{+}}\frac{d}{d\epsilon}& \left( \ep^{\frac{p+4}{(p+1)(p+2)}}\frac{I^{2,{\epsilon}}(z(y),w(y))}{\omega(\epsilon)}\right)\\
& = \frac{2p}{(p+1)(p+2)}  \lim_{\ep \to 0^{+}}\left( \ep^{- \frac{p}{(p+1)(p+2)}-1 }
\left (\int_{\mathbb{R}}   z(y , \ep ) w ( y, \ep )  dy + o(1) \right ) \right) \, .
\end{split}
\end{equation*}
Thus, by \eqref{deri2}, after a straightforward calculation, it is obtained that
\begin{align*}
d''(\omega ) = & \lim_{\ep \to 0^{+}}\Bigg(\left [  \frac{4 \left (\displaystyle{\int_{\mathbb{R}}}   z(y , \ep ) w ( y, \ep )  dy + o(1) \right ) } {pI^\ep ( z(y, \ep ), w (y, \ep ))} - \frac{\omega p} {p+2} + o(1) \right ]
\left ( \frac 2{p+2} \right )^{\frac{2} p } \\
& \times \left ( I^\ep ( z(y, \ep ), w (y, \ep ))\right )^{\frac2 p } \left (\int_{\mathbb{R}}   z(y , \ep ) w ( y, \ep )  dy + o(1) \right ) \ep^{\frac{-3p^2 -2p + 8}{p(p+1)(p+2)}}\big (1 + o(1) \big )\Bigg)
\end{align*}
Therefore, as $\ep \to 0^+$, the sign of $d''(\omega)$ is determined by
\begin{align}
 \frac{4 \left (\displaystyle{\int_{\mathbb{R}}}   w_0^2 ( y )  dy \right ) } {p {\mathcal J}^0} - \frac{ p} {p+2}\, , \label{wwwww1}
\end{align}
where $w_0(x)$ satisfies \eqref{kdv-t} and
\begin{align}
{\mathcal J}^0  = \int_{\mathbb{R} }\left ( w_0^2 + \left ( \sigma -\frac13\right ) w_{0y }^2 \right ) dy.\label{wwwww}
\end{align}

To calculate $w_0(x)$, we note that by a classical theory of ordinary differential equations, \eqref{kdv-t} has a unique homoclinic solution of the form
\begin{align*}
 w_0 (x) = - \left ( {\mathcal J}^0\right )^{-\frac1 p } \left ( \frac4 {(p+1)(p+2) } \right ) ^{-\frac1 p } \mbox{sech}^{\frac2 p } \left ( \frac {px} { 2 \sqrt{ \sigma -1/3}}\right ),
\end{align*}
with an arbitrary translation in $x$. Plug this $w_0$ into \eqref{wwwww} to obtain
\begin{equation}\label{rrr1}
{\mathcal J}_0 = \left ( \frac{ 2 (p+2)^{\frac2 p + 1 } (p+1)^{\frac2 p } \sqrt{\sigma -1/3} \, B (2/p, 2/p) }{p (p+4)} \right) ^ {\frac p{p+2}},
\end{equation}
where $B( x, y)$ is the Beta function of variables $x,y$, and the formula
$$
\int_{\mathbb{R}} \mbox{sech}^{2\nu } (a y ) dy = \frac {2 \cdot 4^{\nu -1} } a B(\nu, \nu )
$$
has been used. Moreover, it can be derived similarly that
\begin{equation}\label{rrrr1}
\int_{\mathbb{R}} w_0^2 dy = \left ( \frac{p+4} 2\right)^{\frac 2 {p+2}} (p+2) ^{\frac 2 p} (p+1)^{ -\frac 4{p(p+2)} }
\left ( \frac{  \sqrt{\sigma -1/3} \, B (2/p, 2/p) }{p} \right) ^ {\frac p{p+2}} \, .
\end{equation}
Hence, putting \eqref{rrr1} and \eqref{rrrr1} into \eqref{wwwww1} gives that the sign of $d''(\omega )$ is determined by the following relation
$$
\frac2 p \left ( p+2 \right )^{\frac 2 p -1} (p+1 ) ^{-  \frac 2 p } (p+4) - \frac p {p+2} = \left (  \left ( \frac {p+2}{p+1} \right ) ^{\frac 2 p } - \frac {p^2} { 2 (p+4) }  \right ) \frac { 2 (p+4) } {p(p+2)}
$$
which has a unique positive root $p_0  $ for
$$
\left ( \frac {p+2}{p+1} \right ) ^{\frac 2 p } - \frac {p^2} { 2 (p+4) } = 0\, ,
$$
since a straightforward computation shows that for $p \geq 1$,
$$
\frac{d}{dp} \left ( \left ( \frac {p+2}{p+1} \right ) ^{\frac 2 p } - \frac {p^2} { 2 (p+4) }\right )  < 0 $$
and $$\left ( \frac {p+2}{p+1} \right ) ^{\frac 2 p } - \frac {p^2} { 2 (p+4) } \  \mbox{is from } + \infty \mbox{ to } -\infty
$$
as $p$ goes from $0$ to $\infty$. Moreover, if $1\leq p \leq 4$,
$$
\left ( \frac {p+2}{p+1} \right ) ^{\frac 2 p } - \frac {p^2} { 2 (p+4) } > 1 - \frac {p^2} { 2 (p+4) } = \frac {(4 - p ) (p+2 )} { 2 (p+4) } \geq 0,
$$
which implies that $p_0 > 4$. Numerically, $p_0$ is approximately equal to $4.2280673976$.
Hence, when $\epsilon$ is  near $0^{+}$, we get
\begin{equation*}
\begin{cases}
d''>0,\quad  \text{for} \quad 1\leq p < p_0 , \\
d''< 0,\quad \text{for} \quad p > p_0,
\end{cases}
\end{equation*}
showing the lemma.
\end{proof}


\section{Stability result}\label{sec5}
In this section, we always assume that $1\leq  p < p_0$ satisfies \eqref{assump_p} so that $d''(\omega )  > 0$, for small $\ep > 0$,  with $\omega $ close enough to $1^-$.  Here, we emphasize that \eqref{defuz} implies that when $\omega \rightarrow 1^-$, $(\psi_\omega, v_\omega)$ is small in  $H^1(\mathbb{R})\times H^1(\mathbb{R})$. Therefore, if the initial condition of \eqref{1bbl} is near $(\psi_\omega, v_\omega)$ for $\omega$ near 1, then the global existence and uniqueness result implies that the solution of \eqref{1bbl} exists for all $ t \in [0, \infty)$. Now, let us introduce some notations.

We denote any pair of function $(\psi, v)$ as an element in $X$, the pair $(\psi_{\omega},v_{\omega})$ as a minimizer of the problem \eqref{mp} with $ {\mathcal G}_\omega$ as the set of such minimizers (local and global), and $(\tilde \psi_\omega ,\tilde v_\omega )$ as the solution of \eqref{trav-eqs} that corresponds to $(\psi_{\omega},v_{\omega})$ in $ {\mathcal G}_\omega$. Thus, from Theorem \ref{wsol}, it is deduced that
\begin{equation}\label{g-tilde}
\widetilde{\mathcal{G}}_{\omega}= \left \{(\tilde \psi_\omega ,\tilde v_\omega )=\left(\frac{2}{p+2}\mathcal{I}_{\omega}\right)^{\frac{1}{p}}(\psi_{\omega},v_{\omega}) :  \   (\psi_{\omega},v_{\omega}) \in {\mathcal G}_\omega \right \}.
\end{equation}
Also, define
$$
U_{\omega,\epsilon}=\{(\psi,v)\in X: \inf_{(\tilde{\psi}_{\omega},\tilde{v}_{\omega})\in  \tilde {\mathcal{G}}_{\omega}}\|(\psi,v)-(\tilde{\psi}_{\omega},\tilde{v}_{\omega})\|_{X}<\epsilon\}.
$$
Since $d(\omega)$ is differentiable ($\hat d (\omega)$ if necessary) and decreasing for $\omega>0$ near to $1^{-}$ (relative to a special $(\tilde{\psi}_{\omega},\tilde{v}_{\omega})\in \tilde {\mathcal{G}}_{\omega}$,
see Lemma \ref{deri_a}), it follows that for $(\psi,v)$ near of $(\tilde{\psi}_{\omega},\tilde{v}_{\omega})\in \tilde {\mathcal{G}}_{\omega}$, we have a $C^{1}$ map
$$\omega(\cdot,\cdot):U_{\omega,\epsilon}\to(0,1), \quad \text{for small $\epsilon>0,$}$$
thanks to relations \eqref{critical2} and \eqref{d_w1rr}, given by
\begin{equation}\label{omega}
\omega(\psi,v)=d^{-1}\left(-\frac{4}{p}G(\psi,v)\right),
\end{equation}
and $\omega(\tilde{\psi}_{\omega},\tilde{v}_{\omega})=\omega,$ for any $(\tilde{\psi}_{\omega},\tilde{v}_{\omega})\in \tilde{\mathcal{G}}_{\omega}$ and with $\omega >0$ near $1^{-}$. The next result uses a variational characterization of such a solution to establish the key inequality to prove the main result.
\begin{lem}\label{prelem} Under the hypothesis of Theorem \ref{main}, there exists $\epsilon>0$ such that for $(\tilde{\psi}_{\omega},\tilde{v}_{\omega})\in \tilde{\mathcal{G}}_{\omega}$ and $(\psi,v)\in U_{\omega,\epsilon}$ relative to this  $(\tilde{\psi}_{\omega},\tilde{v}_{\omega})$, it follows that
$$
\mathcal H(\psi,v)-\mathcal H\left(\tilde{\psi}_{\omega},\tilde{v}_{\omega}\right)+\omega(\psi,v)\left(\mathcal{Q}(\psi,v)-\mathcal{Q}\left(\tilde{\psi}_{\omega},\tilde{v}_{\omega}\right)\right) \geq \frac{1}{4} d^{\prime \prime}(\omega)|\omega(\psi,v)-\omega|^2,
$$
where $\omega(\psi,v)$ is defined by \eqref{omega} for $(\psi,v)\in U_{\omega,\epsilon}$.
\end{lem}
\begin{proof}
We begin by observing, from equations \eqref{JJ_w} and \eqref{d_w1rr}, that
\begin{equation}\label{27r}
\mathcal{H}(\psi, v) + \omega(\psi, v)\mathcal{Q}(\psi, v) = \frac{1}{2} \left( I_{\omega(\psi, v)}(\psi, v) + G(\psi, v) \right).
\end{equation}
Note that
$$ - \frac{4}{p} d(\omega(\psi, v)) = G(\psi, v),$$
and also
$$-\frac{4}{p} d(\omega(\psi, v)) = G(\tilde{\psi}_{\omega(\psi, v)}, \tilde{v}_{\omega(\psi, v)}), \quad (\tilde{\psi}_{\omega(\psi, v)}, \tilde{v}_{\omega(\psi, v)}) \in \tilde{\mathcal{G}}_{\omega(\psi, v)},$$
which implies
$$G(\psi, v) = G(\tilde{\psi}_{\omega(\psi, v)}, \tilde{v}_{\omega(\psi, v)}) .$$
In particular, if we let 
$$ 
c_0 = \left ( \frac{4}{p} d(\omega(\psi, v))\right )^{\frac1{p+2}}
$$
then $ G(\psi /c_0, v/c_0) = -1$, which implies that
$$
I_{\omega(\psi, v )}(\psi/c_0, v/c_0) \geq I_{\omega(\psi, v)}(\psi_{\omega(\psi, v)}, v_{\omega(\psi, v)}),
$$
since $(\psi_{\omega(\psi, v)}, v_{\omega(\psi, v)})$ is a minimizer of $I_{\omega(\psi, v)}$ under the constraint $G(\psi_{\omega(\psi, v)}, v_{\omega(\psi, v)}) = -1$. Hence, it is obtained from the form of $I_\omega$ and the definition of $d(\omega)$ in \eqref{d_w1} with \eqref{d_w} that $c_0 = ((2/(p+2))\mathcal{I}_\omega)^{1/p}$ and 
\begin{equation}\label{Iomega}
I_{\omega(\psi, v)}(\psi, v) \geq I_{\omega(\psi, v)}(c_0\psi_{\omega(\psi, v)}, c_0 v_{\omega(\psi, v)}) = I_{\omega(\psi, v)}(\tilde \psi_{\omega(\psi, v)}, \tilde v_{\omega(\psi, v)})\, ,
\end{equation}
where \eqref{zzzzz} has been used.

On the other hand, the function $d$ satisfies
$$d(\omega) = \mathcal{H}(\tilde{\psi}_{\omega}, \tilde{v}_{\omega}) + \omega \mathcal{Q}(\tilde{\psi}_{\omega}, \tilde{v}_{\omega}),$$
which implies, using the identity in \eqref{deri}, that
$$\mathcal{H}(\tilde{\psi}_{\omega}, \tilde{v}_{\omega}) = d(\omega) - \omega d^{\prime}(\omega).$$
Now,  using that $\omega(\psi, v) \in C^1$, Lemmas \ref{deri_a} and \ref{lem47}, the inequality \eqref{Iomega}, and the identity $$G(\psi, v) = G(\tilde{\psi}_{\omega(\psi, v)}, \tilde{v}_{\omega(\psi, v)}),$$ we deduce from \eqref{27r} that
$$
\begin{aligned}
\mathcal{H}(\psi, v) + \omega(\psi, v)\mathcal{Q}(\psi, v)
&= \frac{1}{2} \left( I_{\omega(\psi, v)}(\psi, v) + G(\psi, v) \right) \\
&\geq \frac{1}{2} \left( I_{\omega(\psi, v)}(\tilde{\psi}_{\omega(\psi, v)}, \tilde{v}_{\omega(\psi, v)}) + G(\tilde{\psi}_{\omega(\psi, v)}, \tilde{v}_{\omega(\psi, v)}) \right) \\
&= d(\omega(\tilde{\psi}_{\omega(\psi, v)}, \tilde{v}_{\omega(\psi, v)})) \\
&= d(\omega(\psi, v)) \\
&\geq d(\omega) + d^{\prime}(\omega)(\omega(\psi, v) - \omega) + \frac{1}{4} d^{{\prime}{\prime}}(\omega) |\omega(\psi, v) - \omega|^2 \\
&= \mathcal{H}(\tilde{\psi}_{\omega}, \tilde{v}_{\omega}) + \omega(\psi, v)\mathcal{Q}(\tilde{\psi}_{\omega}, \tilde{v}_{\omega}) + \frac{1}{4} d^{{\prime}{\prime}}(\omega) |\omega(\psi, v) - \omega|^2,
\end{aligned}$$
where the fifth line follows from Taylor’s expansion at $\omega$, and in the last line we used again the identity \eqref{deri}, that is, $d^{\prime}(\omega) = \mathcal{Q}(\tilde{\psi}_{\omega}, \tilde{v}_{\omega}) $. This concludes the proof.
\end{proof}
With this in hand, let us now prove the main result of the article.

\begin{proof}[Proof of Theorem \ref{main}] First, consider the following: let $U(t)$ be a global solution of the generalized $abcb$-Boussinesq system \eqref{1bbl} in the form
\begin{equation}\label{vetorial}
\begin{cases}
U(t)=(\eta(t),u(t)),&t>0,\\
U(0)=U_0=(\eta(0),u(0))& \text{in }X.
\end{cases}
\end{equation}
Now, suppose that the solution set $\tilde{\mathcal G}_\omega$ is unstable. Then, for a $\tilde U^\omega = (\tilde \psi_\omega, \tilde v _\omega) \in \tilde{\mathcal G}_\omega$,  there exists a sequence of initial data  $\{U_0^k\}_{k\in\mathbb{N}}\subset X$ and $\delta>0$, such that
$$
\lim_{k\rightarrow\infty}\|U_0^k-\tilde U^\omega \|_{X}=0 \quad \text{and} \quad  \
\inf_{\tilde V\in \tilde{\mathcal G}_{{\omega}}}\|U^k(t)-\tilde V\|_{X}\geq\delta\quad \mbox{for some } \ t > 0\, ,
$$
where $U^k$ denotes the sequence of solutions to the system \eqref{vetorial} with initial condition $U^k(0)=U_0^k$. By continuity in $t$, we can pick the first time $t_k$ such that,
\begin{equation}\label{contradiction}
\inf_{\tilde V\in\tilde {\mathcal G}_{{\omega}}}\|U^k(t_k)-\tilde V\|_{X}=\delta>0,
\end{equation}
where at least in the interval $[0,t_k]$ the solution $U^k$ exists.  Moreover, we have that $\mathcal H(U)$ and $\mathcal{Q}(U)$ are conserved at $t$ and continuous for $U(t)=(\eta(t),u(t))$, which implies that
$$
\left|\mathcal H\left(U^k\left(t_k\right)\right)-{\mathcal H}\big(\tilde U^{\omega}\big)\right|=\left|\mathcal H\left(U^k(0)\right)-{\mathcal H}\big (\tilde U^{\omega}\big )\right| \rightarrow 0,
$$
and
$$
\left|\mathcal{Q}\left(U^k\left(t_k\right)\right)-\mathcal{Q}\big(\tilde U^{\omega}\big)\right|=\left|\mathcal{Q}\left(U^k(0)\right)-\mathcal{Q}\big(\tilde U^{\omega}\big)\right| \rightarrow 0,
$$
as $k \rightarrow \infty$. In the following, if necessary, $\hat d''(\omega)$ has to be used instead of $d''(\omega)$. Now, pick $\delta$ small enough so that Lemma \ref{prelem} can be applied, which ensures that
\begin{equation}\label{r35}
\begin{split}
\mathcal H\left(U^k\left(t_k\right)\right)-\mathcal H\big(\tilde U^{\omega}\big)&+\omega\left(U^k\left(t_k\right)\right)\left(\mathcal{Q}\left(U^k\left(t_k\right)\right)-\mathcal{Q}\big(\tilde U^{\omega}\big)\right) \\
&\geq \frac{1}{4} d^{\prime \prime}(\omega)\left|\omega\left(U^k\left(t_k\right)\right)-\omega\right|^2.
\end{split}
\end{equation}
By the fact that $\omega\left(U^k\left(t_k\right)\right)$ is uniformly bounded for $k$, using \eqref{r35} and letting $k \rightarrow \infty$, it is obtained that 
$$
\omega\left(U^k\left(t_k\right)\right) \rightarrow \omega,
$$
and therefore,
\begin{equation}\label{r39}
\lim _{k \rightarrow \infty} G\left(U^k\left(t_k\right)\right)=-\frac{4}{p} \lim _{k \rightarrow \infty} d\left(\omega\left(U^k\left(t_k\right)\right)\right)=-\frac{4}{p}d(\omega).
\end{equation}
On the other hand,
\begin{equation}\label{r40}
\begin{split}
I_\omega\left(U^k\left(t_k\right)\right) +G\left(U^k\left(t_k\right)\right)=&2\left(\mathcal H\left(U^k\left(t_k\right)\right)+\omega \mathcal{Q}\left(U^k\left(t_k\right)\right)\right) \\
=&2\left(\mathcal H\left(U^k\left(0\right)\right)+\omega \mathcal{Q}\left(U^k\left(0\right)\right)\right)\, ,
\end{split}
\end{equation}
since the quantities of the right-hand side are conserved in $t$. Taking $k\to\infty$ in \eqref{r40} and using \eqref{r39}, as $U^{k}(0)$ is the initial data of \eqref{vetorial}, yield that
$$
I_{\omega}\left(U^k\left(t_k\right)\right) \rightarrow 2 d(\omega)+\frac{4}{p} d(\omega)=\frac{2(p+2)}{p} d(\omega)=I_{\omega}(\tilde U^{\omega}).
$$

Let
$$
Z_k\left(t_k\right)=\left(G\left(U^k\left(t_k\right)\right)\right)^{-\frac{1}{p+2}} U^k\left(t_k\right)\, .
$$
Noting that $G\left(Z_k\left(t_k\right)\right)=-1$ and making $k\to\infty$, we have that
\begin{align*}
I_\omega\left(Z_k\left(t_k\right)\right) & =\left(G\left(U^k\left(t_k\right)\right)\right)^{-\frac{2}{p+2}} I_\omega\left(U^k\left(t_k\right)\right) \\
& \to
\left((4/p) d (\omega) \right)^{-\frac{2}{p+2}}I_\omega\big( \tilde U^\omega\big) = I_\omega (\psi_\omega, v_\omega) = {\mathcal I}_\omega\, .
\end{align*}
Hence, $Z_k\left(t_k\right)$ is a minimizing sequence for \eqref{mp}. Therefore, there exists $U_1^\omega\in \mathcal{G}_\omega$ such that, after possible translations and subsequences,
$$
\lim _{k \rightarrow \infty}\left\|Z_k\left(t_k\right)-U_1^\omega\right\|_X=0,
$$
with $G\left(U_1^\omega\right)=-1$. Finally, since $\tilde U_1^\omega \in \tilde {\mathcal G}_\omega $,  the previous limit together with to the fact that
$$
\left\|U^k\left(t_k\right)-\tilde U_1^\omega\right\|_X=\left(G\left(U^k\left(t_k\right)\right)\right)^{\frac{1}{p+2}}\left\|\left(G\left(U^k\left(t_k\right)\right)\right)^{-\frac{1}{p+2}}\left(U^k\left(t_k\right)-\tilde U_1^\omega\right)\right\|_X
$$
gives us
\begin{equation*}
\begin{split}
 \lim _{k \rightarrow \infty}\left\|U^k\left(t_k\right)-\tilde U_1^\omega\right\|_X \leq &M (I_{\omega}(U^{\omega}))^{\frac{1}{p+2}}\lim _{k \rightarrow \infty}\left\|Z_k\left(t_k\right)-(G(U^k)(t_{k}))^{-\frac{1}{p+2}}\tilde U_1^{\omega}\right\|_X \\
  = &M (I_{\omega}(U^{\omega}))^{\frac{1}{p+2}}\lim _{k \rightarrow \infty}\left\|Z_k\left(t_k\right)-\left ((4/p)d(\omega)\right )^{-\frac{1}{p+2}}\tilde U_1^{\omega}\right\|_X \\
   = &M (I_{\omega}(U^{\omega}))^{\frac{1}{p+2}}\lim _{k \rightarrow \infty}\left\|Z_k\left(t_k\right)- U_1^{\omega}\right\|_X   =0,
\end{split}
\end{equation*}
which contradicts \eqref{contradiction}, and the result is shown.
\end{proof}

\appendix
\section{Properties of the transformed system}\label{A}
The goal of this appendix is to prove the properties of the solutions of the system \eqref{trav-ep}. The main result ensures that its derivative with respect to $\epsilon$ is bounded.
\begin{lem}\label{derivative}
The pair $\epsilon^{\frac{p-1}{p+1}}(z_{\epsilon},w_{\epsilon})$, where $(z, w)$ is a solution of \eqref{trav-ep}, is bounded in Sobolev space $H^1(\mathbb{R})$. Here, the subscript $\epsilon$ means the derivative concerning this variable.
\end{lem}

For the sake of simplicity, we will omit $\epsilon$ in the solution forms. Thus, to show this lemma, let us first consider the change of variable  $z=\omega w+\epsilon^{\frac{2}{p+1}}\xi$ and replace the equations in \eqref{trav-ep} by
\begin{equation}
\left\{
\begin{array}{rl}
w-\omega \xi+\omega^{2} bw''+\epsilon^{\frac2{p+1}}\omega b\xi''+aw''+\mathcal I^{\ep}\left(\frac2{p+2}\right)  (\omega w +\epsilon^{\frac2{p+1}}\xi)w^{p}& =0, \\
\\ \label{trav-epa}
\xi+b\omega w''+c\omega w''+c\epsilon^{\frac2{p+1}}\xi''+\mathcal I^{\ep}\left(\frac{2}{(p+1)(p+2)}\right)  w^{p+1}&=0.
\end{array}\right.
\end{equation}
Multiplying the first equation of  \eqref{trav-epa} by $c$ and the second one by $b\omega$ yields that
\begin{equation*}
\left\{
\begin{array}{rl}
cw-c\omega \xi+c\omega^{2} bw''+c\epsilon^{\frac2{p+1}}\omega b\xi''+acw''+c\mathcal I^{\ep}\left(\frac2{p+2}\right)  (\omega w +\epsilon^{\frac2{p+1}}z\xi)w^{p}& =0,
\\ \\
b\omega\xi+\omega^{2} b^{2}w''+c\omega^{2}b w''+c\epsilon^{\frac2{p+1}}b\omega\xi''+b\omega\mathcal I^{\ep}\left(\frac{2}{(p+1)(p+2)}\right)  w^{p+1}&=0.
\end{array}\right.
\end{equation*}
Subtracting in the previous system the first equation with the second one,  we have that
\begin{equation*}
\begin{split}
cw-c\omega\xi+acw''+&c\mathcal I^{\ep}\left(\frac2{ p+2 }\right)\omega w^{p+1} +c\mathcal I^{\ep}\left(\frac2{p+2}\right)\epsilon^{\frac2{p+1}}\xi w^{p}\\&-bw\xi-\omega^{2}b^{2}w''-b\omega\mathcal I^{\ep}\left(\frac{2}{(p+1)(p+2) }\right)  w^{p+1}=0\, ,
\end{split}
\end{equation*}
that is,
\begin{equation*}
\begin{split}
&\xi\left(-\omega(b+c)+c\epsilon^{\frac2{p+1}}\mathcal I^{\ep}\left(\frac2{p+2}\right)w^{p}\right)=-w''(ac-\omega^{2}b^{2}) -cw\\&\qquad\qquad  +\omega\left(b\mathcal I^{\ep}\left(\frac{2}{(p+1)(p+2)}\right)-c\mathcal I^{\ep}\left(\frac2{p+2}\right)\right) w^{p+1}.
\end{split}
\end{equation*}
Therefore,
$$
\xi=\frac{-w''(ac-\omega^{2}b^{2}) -cw +\omega \left(b\mathcal I^{\ep}\left(\frac{2}{(p+1)(p+2)}\right)-c\mathcal I^{\ep}\left(\frac2{p+2}\right)\right)w^{p+1}}{-\omega(b+c)+c\epsilon^{\frac2{p+1}}\mathcal I^{\ep}\left(\frac2{p+2}\right)w^{p}}.
$$
From now on, to make the computation clear, consider the function $\xi$ as
\begin{equation}\label{xi}
\xi=-Aw''+Bw+Cw^{p+1},
\end{equation}
where
$$A:=A(w)=\frac{ac-\omega^{2}b^{2}}{-\omega(b+c)+c\epsilon^{\frac2{p+1}}\mathcal I^{\ep}\left(\frac2{p+2}\right)w^{p}},$$
$$B:=B(w)=\frac{-c}{-\omega(b+c)+c\epsilon^{\frac2{p+1}}\mathcal I^{\ep}\left(\frac2{p+2}\right)w^{p}}\, ,$$
and
$$C:=C(w)=\frac{\omega\left(b\mathcal I^{\ep}\left(\frac{2}{(p+1)(p+2)}\right)-c\mathcal I^{\ep}\left(\frac2{p+2}\right)\right)}{-\omega(b+c)+c\epsilon^{\frac2{p+1}}\mathcal I^{\ep}\left(\frac2{p+2}\right)w^{p}}.$$
Then, differentiating the relation \eqref{xi} twice with respect to $x$ yields that
\begin{equation*}
\begin{split}
\xi''=&-(A''w''+2A'w'''+Aw'''')+(B''w+2B'w'+Bw'')\\&+C''w^{p+1}+2C'(p+1)w^{p}+C(p+1)pw^{p-1},
\end{split}
\end{equation*}
where the superscript $'$ in $A, B$ and $C$ indicates the derivative(s) with respect to $x$ and $A'$ will introduce a factor $\ep^{\frac2{p+2}}$. Hence, replacing \eqref{xi} and $\xi''$ in the second equation of \eqref{trav-epa} gives us
\begin{equation*}
\begin{split}
&-Aw''+Bw+Cw^{p+1}+\omega(b+c)w''+c\epsilon^{\frac2{p+1}}\left(-(A''w''+2A'w'''+Aw'''')\right.\\&+\left.(B''w+2B'w'+Bw'')+C''w^{p+1}+2C'(p+1)w^{p}+C(p+1)pw^{p-1}\right)\\&+\mathcal I^{\ep}\left(\frac{2}{(p+1)(p+2)}\right)  w^{p+1}=0,
\end{split}
\end{equation*}
and arranging similar terms finds that
\begin{equation*}
\begin{split}
&-c\epsilon^{\frac2{p+1}}Aw''''-2\epsilon^{\frac2{p+1}}cA'w'''+\left(-A+\omega(b+c)+c\epsilon^{\frac2{p+1}}(-A''+B)\right)w''\\&+2c\epsilon^{\frac2{p+1}}B'w'+(B+c\epsilon^{\frac2{p+1}}B'')w+C(p+1)pc\epsilon^{\frac2{p+1}}w^{p-1}+
2C'(p+1)c\epsilon^{\frac2{p+1}}w^{p}\\
&+\left(C+c\epsilon^{\frac2{p+1}}C''+\mathcal I^{\ep}\left(\frac{2}{(p+1)(p+2)}\right)\right)w^{p+1}=0,
\end{split}
\end{equation*}
or equivalently,
\begin{equation}\label{A2}
\begin{split}
-c\epsilon^{\frac2{p+1}}Aw''''+(-A+\omega(b+c))w''+B w+\left(C+\mathcal I^{\ep}\left(\frac{2}{p+2}\right)\right)w^{p+1}\\=\epsilon^{\frac{2}{(p+1)(p+2)}}\mathcal{P},
\end{split}
\end{equation}
where
\begin{equation*}
\begin{split}
\mathcal{P}[w]=&2cA'w'''-\left(c(-A''+B)\right)w''-2cB'w'-cB''w\\&-C(p+1)pcw^{p-1}-2C'(p+1)cw^{p}-cC''w^{p+1}.
\end{split}
\end{equation*}
Here, we note that by the assumptions on  $a, b, c$ and the fact that $w $ is uniformly bounded in $H^1(\R)$ for $\ep > 0$ small, the linear part of \eqref{A2} is uniformly invertible, which
implies from the classical theory of ordinary differential equation that $w $ is uniformly bounded in $H^2(\R)$, $\ep^{\frac2{p+1}}w''''\in L^2(\mathbb{R})$ and its norm is uniformly bounded together with $\ep^{\frac1{p+1}}w'''\in L^2(\mathbb{R})$ and its norm being uniformly bounded.

We are now in a position to prove that, for $\epsilon>0$ small enough, $\epsilon^{\frac{p-1}{p+1}}w_{\epsilon}$ and $\epsilon^{\frac{p-1}{p+1}}z_{\epsilon}$ are bounded.

\begin{proof}[Proof of Lemma \ref{derivative}] Rewrite \eqref{A2} in the following form
\begin{equation}\label{A3}
\epsilon^{\frac2{p+1}}\alpha w''''+ \lambda w''+ w+\beta w^{k}=\epsilon^{\frac{2}{p+1}}\mathcal{\tilde{P}}[w],
\end{equation}
with $k=p+1 > 1$ and $\mathcal{\tilde{P}}(\omega,w,w',w'',w''',A',A''',B,B',B'',C',C'',p,c):=\mathcal{\tilde{P}}$, where
\begin{equation*}
\begin{split}
\alpha=&ac-b^2 >0,\quad \lambda = a+b+2b = \frac13 - \sigma < 0,\quad
\beta=\frac2{p+1} {\mathcal I}^{\ep}\, ,
\end{split}
\end{equation*}
and $\mathcal{\tilde{P}}$ is the remaining term and has a similar form with $={\mathcal{P}[w]}$.

We note that $(z, w)$ is the solution of \eqref{trav-ep} and $w$ is a solution of \eqref{A3}. By the theory of ordinary differential equations, due to the symmetry of the equation \eqref{A3} with respect to $x$, it can be deduced that
any solution of \eqref{A3} is even in $x$ after a suitable translation in $x$-variable. Moreover, if $\ep > 0$, the equation \eqref{trav-ep} or \eqref{A3} is differentiable with respect to $\ep$, which implies that the solution $w$ is also differentiable with respect to $\ep$. Now, taking the derivative in terms of $\epsilon$  on both sides of \eqref{A3} and, after that,  multiplying the result by $\epsilon^{\frac{p-1}{p+1}}$, we have
\begin{equation}\label{A3a}
\begin{split}
\epsilon\alpha w''''_{\epsilon}+ \epsilon^{\frac{p-1}{p+1}}\lambda w''_{\epsilon}+  \epsilon^{\frac{p-1}{p+1}}w_{\epsilon}+\beta \epsilon^{\frac{p-1}{p+1}}kw^{k-1}  w_{\epsilon}=&\mathcal{\tilde{\tilde{P}}}_1[w, \epsilon] L [\epsilon w_\ep] + \mathcal{\tilde{\tilde{P}}}_2[w, \epsilon] \\=& \mathcal{\tilde{\tilde{P}}}[w, w_\ep, \epsilon],
\end{split}
\end{equation}
where $\mathcal{\tilde{\tilde{P}}}_1[w, \epsilon] $ and $ \mathcal{\tilde{\tilde{P}}}_2[w, \epsilon]$ are functions that only depend on $w$ and $\epsilon$ and are uniformly bounded in any Sobolev norms as $\ep$ small and $\ep\to 0$, and $L [\epsilon w_\ep]$ is linear in terms of $\epsilon w_\ep$ or its $x$-derivatives . Let $\hat w = \epsilon^{\frac{p-1}{p+1}} w_\epsilon $, which changes \eqref{A3a} to
\begin{equation}\label{A3aa}
\begin{split}
\epsilon^{\frac 2 {p+1} } \alpha \hat w''''+ \lambda \hat w'' +  \hat w+\beta kw^{k-1}  \hat w=&\mathcal{\tilde{\tilde{P}}}_1[w, \epsilon] \epsilon^{\frac 2 {p+1} }  L [\hat w] + \mathcal{\tilde{\tilde{P}}}_2[w, \epsilon] \\=& \mathcal{\tilde{\tilde{P}}}[w, \hat w, \epsilon].
\end{split}
\end{equation}

Consider the linear equation associated with \eqref{A3aa},
$$
\epsilon^{\frac 2 {p+1} } \alpha \hat w'''' - |\lambda| \hat w'' +  \hat w=0\, .
$$
The characteristic equation associated with the linear equation is
$$
\alpha \epsilon^{\frac2{p+1}} r^4-|\lambda | r^2 +1=0,
$$
with roots $\pm r_1$ and $\pm r_2 $ and
$$
r_1=\sqrt{\frac{|\lambda |-\sqrt{|\lambda |^2 -4\alpha\epsilon^{\frac2{p+1}}}}{2\alpha\epsilon^{\frac2{p+1}}}} \quad\text{and}\quad
r_2=\sqrt{\frac{|\lambda |+\sqrt{|\lambda |^2 -4\alpha\epsilon^{\frac2{p+1}}}}{2\alpha\epsilon^{\frac2{p+1}}}} \, ,
$$
satisfying $$r_1^2+r_2^2= \frac{1}{\alpha \epsilon^{\frac2{p+1}}}.$$ Using a variation of parameters, the bounded even solutions of \eqref{A3aa} can be written as follows
\begin{equation*}
\begin{split}
\hat w(x)= & -\frac{1}{2 \alpha r_1 \epsilon^{\frac2{p+1}} \left(r_1^2+r_2^2\right)} \int_{-\infty}^{+\infty} e^{-r_1|x-\xi|}\left(-\beta  kw^{k-1}(\xi)\hat w(\xi)+ \mathcal{\tilde{\tilde{P}}}[w, \hat w, \epsilon](\xi)\right) d \xi \\
& -\frac{1}{2 \alpha r_2 \epsilon^{\frac2{p+1}} \left(r_1^2+r_2^2\right)} \int_{-\infty}^{+\infty} e^{-r_2|x-\xi|}\left(-\beta  k w^{k-1}(\xi)\hat w(\xi)+ \mathcal{\tilde{\tilde{P}}}[w, \hat w, \epsilon](\xi)\right) d \xi.
\end{split}
\end{equation*}
By differentiating the previous equality twice with respect to $x$, we obtain
$$
\begin{aligned}
- |\lambda| ( \hat w_{x x}-r_1^2\hat w ) =&  \frac{ - \beta kw^{k-1}(x)\hat w (x)}{ \alpha \epsilon^{\frac2{p+1}}\left(r_1^2+r_2^2\right)}
+\frac{\mathcal{\tilde{\tilde{P}}}(x)}{\alpha \epsilon^{\frac2{p+1}}\left(r_1^2+r_2^2\right)}\\
& +\frac{1}{2 \alpha r_2 \epsilon^{\frac2{p+1}}\left(r_1^2+r_2^2\right)} \int_{-\infty}^{+\infty} e^{-r_2|x-\xi|}\left(\beta
\left(\left(kw^{k-1}(\xi)\hat w(\xi)\right)_{\xi \xi}\right.\right.\\&\left.\left.\qquad\qquad\qquad\qquad\qquad\qquad\qquad\qquad\qquad-r_1^2 kw^{k-1}(\xi)\hat w(\xi)\right)\right) d \xi \\
&+\frac{1}{2\alpha  r_2 \epsilon^{\frac2{p+1}}\left(r_1^2+r_2^2\right)} \int_{-\infty}^{+\infty} e^{-r_2|x-\xi|}
\mathcal{\tilde{\tilde{P}}}(\xi)d \xi \\&-\frac{r_2^2-r_1^2}{2 r_2 \epsilon^{\frac2{p+1}} \left(r_1^2+r_2^2\right)} \int_{-\infty}^{+\infty} e^{-r_2|x-\xi|} \mathcal{\tilde{\tilde{P}}}(\xi) d \xi\\
=&-\left(\frac{\beta   k}{\alpha \epsilon^{\frac2{p+1}}\left(r_1^2+r_2^2\right)}\right) w^{k-1}(x)\hat w(x)+ f_1(x)\\=& -\beta   k w^{k-1}(x)\hat w(x)+ f_1(x)\, .
\end{aligned}
$$
Here, we remark that the solution of \eqref{A3}, which goes to zero at infinity, can also be rewritten in the above form of a second-order integro-differential equation. Then, such a solution must be even in terms of $x$ after a translation, which is
another proof of evenness of $(z, w)$ for \eqref{trav-ep}. Now, we rewrite the above $\hat w$ equation as
\begin{equation}\label{A7}
\begin{split}
&  \hat w_{x x}- |\lambda |^{-1} \hat w - |\lambda|^{-1}  \beta _0  k w_0^{k-1}(x)\hat w (x)= \left ( r_1^2 -|\lambda |^{-1} \right ) \hat w \\
&\qquad - |\lambda|^{-1}  k \left (\beta_0   w_0^{k-1}(x)  -\beta   w^{k-1}(x) \right ) \hat w (x) -|\lambda |^{-1} f_1(x) = f_2(x)\, ,
\end{split}
\end{equation}
where $\beta_0 = \frac2{p+1} {\mathcal J}^0$, $w_{0}(x)$ is a solution of the generalized KdV equation (see Theorem \ref{conver}), and
$$
\left | r_1^2 -|\lambda |^{-1} \right | + \left |\beta_0 w_0^{k-1}(x)  -\beta  w^{k-1}(x) \right | \to 0
$$
as $\ep \to 0$. We mention  that the terms involving $\hat w$ or its derivatives in $f_2(x)$ are linear in terms of $\hat w$ or its derivatives and the coefficients,
which may have $\ep$ or $w$ or $w_0$, will go to zero as $\ep \to 0$.   Since $w$ is even in $x$, the following claim is verified:
\begin{claim}\label{ID} \eqref{A7} can be transformed into an integro-differential equation as:
\begin{equation}\label{A10}
\begin{aligned}
\hat w(x) & =\Xi_1(x) \int_0^x \Xi_2(s) f_2(s) d s+\Xi_2(x) \int_x^{\infty} \Xi_1(s) f_2(s) d s \\
& =\int_0^{+\infty} K(x, s) f_2(s) d s
=\mathcal{L}_{\epsilon}\left[f_2\right](x),
\end{aligned}
\end{equation}
for $x \geq 0$, which can be evenly extended to $x<0$, for appropriated functions $\Xi_1$ and $\Xi_2$.
\end{claim}

Indeed, note that
$$
\Xi_1(x)=\frac{d}{dx} w_{0}
$$ is an odd solution of the homogenous equation for \eqref{A7} with $w_{0}(x)$ as a solution of the generalized KdV equation  such that $\Xi_1 \rightarrow-\exp \left(-|\lambda|^{-1/2} |x|\right)$ as $|x| \rightarrow \infty$. Using the Liouville formula, we have the existence of an even function $\Xi_2(x)$ such that $\{\Xi_1(x),\Xi_2(x)\}$ form a fundamental set of solutions to \eqref{A7} with Wronskian $W\left[\Xi_1, \Xi_2\right]=\Xi_1(x) \Xi_2^{\prime}(x)-\Xi_1^{\prime}(x) \Xi_2(x)=1$.
Therefore, by constructing Green's function $K(x, s)$ using $\Xi_1$ and $\Xi_2$, \eqref{A7} can be transformed into the integral equation \eqref{A10}, giving Claim \ref{ID}.
\vspace{0.1cm}

With this claim, we can apply the contraction mapping theorem to the integro-differential equation \eqref{A10}. To do this, we let the Banach space be the Sobolev space $H^1(\mathbb{R})$ with the corresponding Sobolev norm and define
$$
B_{1}=\left\{f(x) \in H^1(\mathbb{R}) \mid f(-x)=f(x),\ \|f\|_{H^1(\mathbb{R}) } = \|f\|_{B_1 }<\infty\right\} .
$$
Then, applying a similar proof as done in \cite[Section 3]{SunChen}, the following estimate holds for \eqref{A10}.
\begin{lem}\label{SC}
If $f(x) \in B_{1}$, then $$\mathcal{L}[f](x) \in B_1\quad\text{and}\quad \|\mathcal{L}[f](x)\|_{B_1} \leq  \tilde C\|f\|_{B_1}\, ,$$
where $\tilde C$ is independent of $\ep$.
\end{lem}

Now, apply Lemma \ref{SC} to \eqref{A10} together with the uniform boundedness of $w$ in $H^2(\mathbb{R})$ and $\ep^{\frac 1 {p+1}} w'''$ in $L^2(\mathbb{R})$, and the properties of $f_2$ to obtain that if $\hat w\in B_1$,
\begin{align*}
{\mathcal{L}}[f_2 ](x) \in B_1\quad\text{and}\quad \left\| {\mathcal{L}}[f_2 ](x)\right\|_{B_1} \leq C_0(\epsilon)  \left\| \hat w \right\|_{B_1}  + C_1\, ,
\end{align*}
where for small $\ep > 0$, $C_1 > 0 $ is a constant and $C_0 (\ep )\to 0$ as $\ep \to 0$.
Finally, for $s \geq 2 C_1 > 0$ large, consider a closed convex subset of $B_1$ given by
$$
\mathcal{S}_{s}=\left\{\hat w \in B_1 : \|\hat w\|_{B_1} \leq s \right\}.
$$
Then if $\hat w\in \mathcal{S}_s$, we can let $\ep$ small enough such that $C_0 (\ep ) s < C_1$, which implies that ${\mathcal{L}}[f_2 ](x)$ maps $\mathcal{S}_{s}$ to $\mathcal{S}_{s}$. If we let $f^{(j)} _2 (x) $ be the corresponding
$f_2(x) $ for $\hat w^{(j)} (x)\in B_1$, since $\hat w$ is linear in $f_2(x)$, it is straightforward to see that from Lemma \ref{SC} again, we have
$$
\left\| {\mathcal{L}}\left[ f_2^{(1)}\right](x) - {\mathcal{L}}\left[f_2^{(2)}\right](x)\right\|_{B_1} \leq C _0(\ep ) \left\|\hat w^{(1)}(x)-\hat w^{(2)}(x)\right\|_{B_1}\, .
$$
Hence, for small $\ep >0$, it is deduced that ${\mathcal{L}}[f_2 ](x)$ is a contraction for $\hat w \in \mathcal{S}_{s}$ and the contraction mapping principle implies that
$\hat w $ is the only fixed point of ${\mathcal{L}}[f_2 ](x)$ in $\mathcal{S}_{s}$. Therefore,  $\hat w$ in \eqref{A3aa} satisfies that for small $\ep > 0$, $\| \hat w \|_{H^1(\mathbb{R})} \leq s$
where $s$ is independent $\ep$. Since $\hat w = \ep^{\frac{p-1}{p+1}} w_\ep $ and the relation between $\xi, z $ and $w$ is given in \eqref{trav-epa} and \eqref{xi}, it is obtained that $\ep^{\frac{p-1}{p+1}}(z_\ep,  w_\ep ) $ is
uniformly bounded in $H^1(\mathbb{R})$ with respect to small $\ep > 0$, showing Lemma \ref{derivative}.
\end{proof}

\section{The proof of Claim I in the proof of Lemma \ref{limite}}\label{B}

This appendix gives the proof of Claim I, which is stated in the middle of the proof for Lemma \ref{limite}. Here, the concentration-compactness argument from Theorem \ref{lions} can be applied to generate a convergence subsequence in $H^1 (\mathbb{R}) \times H^1 (\mathbb{R})$ and we mainly use the argument for a system in \cite[Section 3.1]{DPDE2013}.

First, let us state the properties of $\left (z^{\ep_j} , w^{\ep_j} \right )$. It is known that
there is a sequence $\{\ep_j\}\rightarrow 0^+ $ such that
$$
\lim_{j \rightarrow\infty } {\mathcal{I}}^{\epsilon_j} = \lim_{j \rightarrow\infty } I^{\ep_j} \left (z^{\ep_j} , w^{\ep_j} \right ) = I_c = \liminf_{\ep \rightarrow 0^+} {\mathcal{I}}^\epsilon \leq {\mathcal{J}}^0 \quad \mbox{with} \quad \ G\left (z^{\ep_j} , w^{\ep_j} \right ) = -1,
$$
where $I^{\ep} \left (z , w \right )$ is defined in \eqref{I-ep}, ${\mathcal{I}}^{\epsilon}$ is the infimum of $I^{\ep} \left (z , w \right )$ under the condition $G\left (z , w \right ) = -1$, and ${\mathcal{J}}^0$ is finite. Hence,
$\left (z^{\ep_j} , w^{\ep_j} \right ), j = 1, 2 , \dots$ are minimizers of $I^{\ep_j} (z , w)$ with $G\left (z , w \right ) = -1$. Now, we apply the concentration-compactness argument (see Theorem \ref{lions}) to this sequence $\left (z^{\ep_j} , w^{\ep_j} \right )$. Since $I^{\ep} \left (z , w \right )$ is non-negative and the limit of $ I^{\ep_j} \left (z^{\ep_j} , w^{\ep_j}\right )$ as $j\rightarrow \infty$ exists, Theorem \ref{lions} can be applied. Here, the positive measure $\{ \nu_j \}$ is defined by $d \nu_j = \rho_j dx$,  where $\rho_j $ is given by
\begin{align*}
\rho_j = & {\ep_j}^{-\frac2{p+1}}\left (z^{\ep_j}-\omega({\ep_j}) w^{\ep_j}\right )^2+ \left ( w^{\ep_j}\right )^2  \\
&\quad + |c|\left(\left (z^{\ep_j}\right ) '  -  \frac{b\omega({\ep_j})}{|c|} \left ( w^{\ep_j}\right ) '\right)^2 +\left(\frac{ac-b^2 \omega^2({\ep_j})}{|c|}\right) \left (\left ( w^{\ep_j}\right )'\right )^2\, ,
\end{align*}
which is the integrand of $I^{\ep_j} \left (z^{\ep_j} , w^{\ep_j} \right )$.

\medskip

\noindent{\bf  i. Vanishing:}
\smallskip

This case can be easily ruled out. If ``vanishing" happened, then $G\left (z^{\ep_j} , w^{\ep_j} \right )$ would approach to zero as $j \rightarrow \infty$, which contradicts to  $G\left (z^{\ep_j} , w^{\ep_j} \right ) = -1$. A detailed proof was given in \cite[Lemma 3.2]{DPDE2013} for an almost identical argument.

\medskip

\noindent{\bf  ii. Dichotomy:}

\smallskip
To rule out ``dichotomy", we follow the usual steps if ``dichotomy" happens. Following the ideas in  the proof of \cite[Lemma 3.4]{DPDE2013}, we let a fixed function $\phi (x) \in C^\infty_0 (\R ) $ such that $\operatorname{supp}( \phi) \subset [-2, 2] $ and $\phi \equiv 1$ in $[-1, 1]$. From the assumptions of ``dichotomy" with $0 < \theta < I_c$, we can choose sequences $\gamma_j \to 0, R_j \to \infty$ such that
$$
\operatorname{supp} \, (\nu_{j}^1) \subset B_{R_j } ( x_j ) \, ,\quad \operatorname{supp} \, (\nu_{j}^2 ) \subset \R \setminus B_{2R_j } ( x_j ) \, ,
$$
and
$$
\limsup_{j \to \infty} \left ( \left | \theta - \int_\R d \nu_j ^1 \right | + \left | (I_c - \theta)  - \int_\R d \nu_j ^2 \right | \right ) = 0\, ,
$$
which implies that
$$
\limsup_{j \to \infty} \left (  \int_{ B_{2R_j } ( x_j ) \setminus  B_{R_j } ( x_j ) }  d \nu_j  \right ) = 0\, .
$$
Based on those properties,  if we let $\phi_j  (x) = \phi ( ( x-x_j ) /R_j) $, then we can establish a splitting for the sequence $\left (z^{\ep_j} , w^{\ep_j} \right )$ by
$$
\left (z^{\ep_j} , w^{\ep_j} \right ) = \left (z^{\ep_j}_1 , w^{\ep_j}_1 \right ) + \left (z^{\ep_j}_2 , w^{\ep_j}_2 \right )
$$
with
$$
 \left (z^{\ep_j}_1 , w^{\ep_j}_1 \right )  = \left (z^{\ep_j} , w^{\ep_j} \right ) \phi_j,\, ,\quad  \left (z^{\ep_j}_2 , w^{\ep_j}_2 \right )= \left (z^{\ep_j} , w^{\ep_j} \right )( 1 - \phi_j)\, ,
$$
and show that as $j \to \infty$,
\begin{align*}
& I^{\ep_j} \left (z^{\ep_j} , w^{\ep_j} \right )  = I^{\ep_j} \left (z^{\ep_j}_1 , w^{\ep_j}_1 \right )+  I^{\ep_j} \left (z^{\ep_j}_2 , w^{\ep_j}_2 \right ) + o (1)\, ,\\
& G\left (z^{\ep_j} , w^{\ep_j} \right ) = G\left (z^{\ep_j}_1 , w^{\ep_j}_1 \right ) + G\left (z^{\ep_j}_2 , w^{\ep_j}_2 \right ) + o(1)\, .
\end{align*}
The proof of the splitting properties is referenced to the same proof in \cite[Lemma 3.3]{DPDE2013}.  Then, by a same proof as that in \cite[Lemma 3.4]{DPDE2013}, it is obtained that
\begin{align*}
& \lim_{j \to \infty} \left ( I^{\ep_j} \left (z^{\ep_j} , w^{\ep_j} \right )  - I^{\ep_j} \left (z^{\ep_j}_1 , w^{\ep_j}_1 \right )-  I^{\ep_j} \left (z^{\ep_j}_2 , w^{\ep_j}_2 \right ) \right ) = 0 \, ,\\
& \lim_{j \to \infty} \left ( G\left (z^{\ep_j} , w^{\ep_j} \right ) - G\left (z^{\ep_j}_1 , w^{\ep_j}_1 \right )  - G\left (z^{\ep_j}_2 , w^{\ep_j}_2 \right )\right ) = 0 \, .
\end{align*}
Let $\lambda_{{\ep_j}, i} = \left | G \left (z^{\ep_j}_i , w^{\ep_j}_i \right )\right | $ for $i = 1,2$. We show that $\lambda_i = \lim_{j \to \infty} \lambda_{{\ep_j}, i}  \not = 0$. If one is those limits is zero, without loss of generality, let
$\lambda_1 = 0 $, which implies that $\lambda_2 = 1$. Then, consider
$$
\left ( \tilde z^{\ep_j}_2 , \tilde w^{\ep_j}_2 \right ) = \lambda_{{\ep_j}, 2} ^{ - \frac1 {p+2} }\left (z^{\ep_j}_2 , w^{\ep_j}_2 \right )\, ,
$$
so that $G ( \left ( \tilde z^{\ep_j}_2 , \tilde w^{\ep_j}_2 \right ) = -1$.  By the construction of $\left (z^{\ep_j}_1 , w^{\ep_j}_1 \right )$, it is deduced that
\begin{align*}
I_c = & \lim_{j \to \infty} \left ( I^{\ep_j} \left (z^{\ep_j}_1 , w^{\ep_j}_1 \right ) +  I^{\ep_j} \left (z^{\ep_j}_2 , w^{\ep_j}_2 \right ) \right ) \\
\geq  &\lim_{j \to \infty} \left ( \int_{B_{R_j} ( x_j )} d \nu_j  + \lambda_{{\ep_j}, 2} ^{\frac 2 {p+2} }  I^{\ep_j} \left (\tilde z^{\ep_j}_2 , \tilde w^{\ep_j}_2 \right )  \right ) \\
\geq &\lim_{j \to \infty} \left ( \int_{\R } d \nu_j^1  + \lambda_{{\ep_j}, 2} ^{\frac 2 {p+2} }  I^{\ep_j} \left (z^{\ep_j} , w^{\ep_j} \right )  \right ) = \theta + I_c,
\end{align*}
where the fact that $I^{\ep_j} \left (z^{\ep_j} , w^{\ep_j} \right )$ is the minimum of $I^{\ep_j} \left (z , w  \right )$ with $G ( z, w) = -1$ has been used. Since $\theta > 0$, the above inequality gives a contradiction.
Thus, $\lambda_i \not =0 $ for $i = 1,2$. Hence, we can define
$$
\left ( \tilde z^{\ep_j}_i , \tilde w^{\ep_j}_i \right ) = \lambda_{{\ep_j}, i} ^{ - \frac1 {p+2} }\left (z^{\ep_j}_i , w^{\ep_j}_i \right )\qquad \mbox{for }\   i = 1,2\, .
$$
which gives $G ( \left ( \tilde z^{\ep_j}_i , \tilde w^{\ep_j}_i \right ) = -1$ (here we note that $G  \left ( \tilde z^{\ep_j}_i , \tilde w^{\ep_j}_i \right ) = \pm 1$. Since $G \left (  z^{\ep_j} ,  w^{\ep_j} \right ) = -1$,
then for $j $ large, $G  \left (  z^{\ep_j}_i ,  w^{\ep_j}_i \right ) $ must be nonzero and negative). Moreover,
\begin{align*}
I_c = & \lim_{j \to \infty} \left ( I^{\ep_j} \left (z^{\ep_j}_1 , w^{\ep_j}_1 \right ) +  I^{\ep_j} \left (z^{\ep_j}_2 , w^{\ep_j}_2 \right ) \right ) \\
= &\lim_{j \to \infty} \left (\lambda_{{\ep_j}, 1} ^{\frac 2 {p+2} }  I^{\ep_j} \left (z^{\ep_j}_1 , w^{\ep_j}_1 \right ) + \lambda_{{\ep_j}, 2} ^{\frac 2 {p+2} }  I^{\ep_j} \left (\tilde z^{\ep_j}_2 , \tilde w^{\ep_j}_2 \right )  \right ) \\
\geq &\lim_{j \to \infty} \left (\lambda_{{\ep_j}, 1} ^{\frac 2 {p+2} }  I^{\ep_j} \left (z^{\ep_j} , w^{\ep_j} \right )  + \lambda_{{\ep_j}, 2} ^{\frac 2 {p+2} }  I^{\ep_j} \left (z^{\ep_j} , w^{\ep_j} \right )  \right )  \\
\geq & \left (\lambda_{ 1} ^{\frac 2 {p+2} }    + \lambda_{ 2} ^{\frac 2 {p+2} }   \right ) I_c
\end{align*}
where, again, the fact that $I^{\ep_j} \left (z^{\ep_j} , w^{\ep_j} \right )$ is the minimum of $I^{\ep_j} \left (z , w  \right )$ with $G ( z, w) = -1$ has been used.
Hence, $1 \geq \left (\lambda_{ 1} ^{\frac 2 {p+2} }    + \lambda_{ 2} ^{\frac 2 {p+2} }   \right )$ with $\lambda_{ i} > 0, j=1,2,$ and $\lambda_1 + \lambda_2 = 1$, which contradicts to the strictly concave property of the function $\lambda^{\frac 2 {p+2} }$ for $p \geq 1$. Therefore, ``dichotomy" is ruled out.

\medskip

\noindent{\bf iii. Compactness:}
\smallskip

Finally, by Theorem \ref{lions}, only ``compactness" holds. Then, in the following, we show that there is a subsequence of $\left (z^{\ep_j} , w^{\ep_j} \right )$ (which we still denote the same), a sequence of points $\{ x_j \} \in \mathbb{R}$, and  $ (z_0, w_0 )\in H^1( \R ) \times H^1( \R )$,  such that the translated functions
$$
\left (\tilde z^{\ep_j} , \tilde w^{\ep_j} \right ) = \left (z^{\ep_j} (\cdot + x_j ) , w^{\ep_j}(\cdot + x_j ) \right )
$$
converge to $ (z_0, w_0 )$ strongly in $ H^1 ( \R )\times H^1( \R )$. The proof is similar to the proof of \cite[Theorem 3.3]{DPDE2013} with some modifications.

It is known that
$$
\lim_{j \rightarrow\infty } I^{\ep_j} \left (z^{\ep_j} , w^{\ep_j} \right ) = I_c \qquad \mbox{and}\qquad G\left (z^{\ep_j} , w^{\ep_j} \right ) =-1 \, .
$$
``Compactness" implies that there is a sequence $\{ x_j \} \in \mathbb{R}$ such that for a given $\gamma  > 0$, there exists an $R > 0$ satisfying
$$
\int_{B_R( x_j )} d \nu_j \geq I_c - \gamma  \quad \mbox{for all} \quad j = 1, 2, \dots .
$$
Define
$$
\tilde \rho_j (x) = \rho_j (x+x_j )\, ,\qquad \left (\tilde z^{\ep_j} (x) , \tilde w^{\ep_j} (x) \right ) = \left (z^{\ep_j} (x + x_j ) , w^{\ep_j}(x + x_j ) \right )\, ,
$$
which have the same properties as $\left (z^{\ep_j}  , w^{\ep_j} \right )$ with
$$
\int_{B_R( 0 )} \tilde \rho_j (x)  dx = \int_{B_R( x_j )} d \nu_j \geq I_c - \gamma  \quad \mbox{for all} \quad j = 1, 2, \dots \, ,
$$
or
\begin{align}
\int_{\mathbb{R} \setminus B_R( 0 )} \tilde \rho_j (x)  dx = \int_{\mathbb{R} \setminus B_R( x_j )} d \nu_j \leq 2 \gamma  \quad \mbox{for all} \quad j = 1, 2, \dots \, . \label{B1}
\end{align}
Since $\left (\tilde z^{\ep_j} (x) , \tilde w^{\ep_j} (x) \right )$ is uniformly bounded in $H^1( \R )\times H^1( \R )$, Sobolev imbedding  theorem shows that there is a subsequence
of $\left (\tilde z^{\ep_j} (x) , \tilde w^{\ep_j} (x) \right )$ (denoted by the same notations) and $(z_0, w_0) \in H^1( \R ) \times H^1( \R )$ such that  as $j \to \infty$,
\begin{align*}
& \left (\tilde z^{\ep_j}  , \tilde w^{\ep_j}  \right ) \rightharpoonup (z_0, w_0) \quad \mbox{in}\quad H^1( \R ) \times H^1( \R ) \quad\mbox{ and } \quad L^2 ( \R ) \times L^2( \R )\, ,\\
& \left (\tilde z^{\ep_j} , \tilde w^{\ep_j}  \right ) \rightarrow  (z_0, w_0) \quad \mbox{in} \quad L^2_{loc} ( \R )\times L^2_{loc}( \R )\, ,\\
& \left (\tilde z^{\ep_j}  , \tilde w^{\ep_j}  \right ) \rightarrow (z_0, w_0) \quad \mbox{a.e.\  \   \   in}\quad \mathbb{R}^2 \, .
\end{align*}
Then, it is deduced from \eqref{B1} that
\begin{align*}
\int_{\mathbb{R} } |z_0 (x) |^2  dx \leq & \liminf _{j \to \infty} \int_{\mathbb{R} } |\tilde z^{\ep_j} (x) |^2  dx  \\
\leq & \liminf _{j \to \infty} \int_{B_R(0) } |\tilde z^{\ep_j} (x) |^2  dx + C \gamma  \\
= &  \int_{B_R(0) } | z_0 (x) |^2  dx + C \gamma \leq \int_{\mathbb{R} } | z_0 (x) |^2  dx + C \gamma\, ,
\end{align*}
where $C > 0$ is a fixed constant which may depend on the constants in $I^{\ep_j} \left (z , w \right )$, but independent of $\gamma$. Hence,
$$
\liminf _{j \to \infty} \int_{\mathbb{R} } |\tilde z^{\ep_j} (x) |^2  dx = \int_{\mathbb{R} } |z_0 (x) |^2  dx
$$
By the weak convergence of $\tilde z^{\ep_j}$ to $z_0 $ in $L^2( \R )$, there is a subsequence of $\tilde z^{\ep_j}$ (still denoted same) such that $\tilde z^{\ep_j}\rightarrow z_0 $ strongly in $L^2( \R )$. Similar argument works for $\tilde w^{\ep_j}\to w_0 $ strongly in $L^2( \R )$. Hence, the uniform boundedness of $I^{\ep_j} \left (z^{\ep_j} , w^{\ep_j} \right )$ yields that $\tilde z^{\ep_j} - \tilde w^{\ep_j}\to 0$ in $L^2( \R )$ and $w_0 = z_0$.
Then, the Sobolev embedding theorem implies that
$$
G (z_0, w_0 ) = \lim_{j \to \infty} G \left (\tilde z^{\ep_j}  , \tilde w^{\ep_j}  \right )  = -1 \, ,
$$
which gives
$$
\lim_{j\to \infty}  I^{\ep_j} \left (z_0 , w_0 \right ) \geq  \lim_{j\to \infty}  I^{\ep_j} \left (z^{\ep_j} , w^{\ep_j} \right ) = I_c\, .
$$
Moreover, for $\ep > 0$ small, the weak convergence of $\left (\tilde z^{\ep_j}  , \tilde w^{\ep_j}  \right )$ to  $(z_0, w_0)$ in $H^1( \R )\times H^1( \R )$ yields that if we denote
\begin{align*}
II^{\ep}(z, w)= &\int_{\mathbb R}\left((z-\omega(\ep) w)^2+(1-\omega^2(\ep))\epsilon^{-\frac2{p+1}}w^2 \right)dy \\
&+\int_{\mathbb R}\left(|c|\left(z'  -  \frac{b\omega(\ep)}{|c|} w'\right)^2 +\left(\frac{ac-b^2 \omega^2(\epsilon)}{|c|}\right) (w')^2\right)\,dy\, ,
\end{align*}
then
\begin{align*}
0  & \leq \lim_{j\to \infty} II^{\ep_j }\left (\tilde z^{\ep_j} -z_0  , \tilde w^{\ep_j}-w_0  \right )
 = \lim_{j\to \infty} \left ( II^{\ep_j }\left (\tilde z^{\ep_j}   , \tilde w^{\ep_j}  \right ) - II^{\ep_j }\left (z_0  , w_0  \right ) \right )\\
&\leq \lim_{j\to \infty} \left ( I^{\ep_j }\left (\tilde z^{\ep_j}   , \tilde w^{\ep_j}  \right ) - II^{\ep_j }\left (z_0  , w_0  \right )\right )  = I_c - \lim_{j\to \infty} I^{\ep_j }\left (z_0  , w_0  \right ) \leq 0
\end{align*}
where the facts that $z_0 = w_0$ in $L^2( \R )$ and $ 1 - \omega(\ep ) =O (\ep^{\frac 2{p+2}})$ have been used. Therefore,
$$\lim_{j\to \infty} II^{\ep_j }\left (\tilde z^{\ep_j} -z_0  , \tilde w^{\ep_j}-w_0  \right ) = 0 \, ,$$
which, together with $L^2$-convergence of $\left (\tilde z^{\ep_j}   , \tilde w^{\ep_j}  \right )$ to $\left (z_0  , w_0  \right ) $, yields that $\left (\tilde z^{\ep_j}   , \tilde w^{\ep_j}  \right )\to \left (z_0  , w_0  \right ) $ in $H^1( \R )\times H^1( \R )$. The claim is proved. $\qquad \Box$

\section{Global existence of \eqref{1bbl} and  proof of Lemma \ref{prepare}}\label{C}

In this appendix, we will briefly give the proofs of global existence and uniqueness of \eqref{1bbl} and Lemma \ref{prepare}.

\subsection{Global well-posedness of \eqref{1bbl}} We note that for $p = 1$, the local well-posedness of \eqref{1bbl} are given in  \cite[Theorem 2.5]{BCS04} and the global existence is provided in \cite[Theorem 4.2]{BCS04} using the Hamiltonian structure of \eqref{1bbl}. For general $p \geq 1$, we will give a very brief account on the proof of the local and global results. 

For local well-posedness of \eqref{1bbl}, the same procedure of the proof in \cite[Theorem 4.2]{BCS04} will be followed. Using the same notations as those in \cite{BCS04}, \eqref{1bbl} is equivalent to 
$$
\begin{pmatrix} v \\w \end{pmatrix} = S(t) \begin{pmatrix} v_0 \\w_0 \end{pmatrix} \int^t_0 S(t-s) \mathcal{F} \begin{pmatrix} v \\w \end{pmatrix} ds\, ,
$$
where only difference between the terms in the system of \cite{BCS04} and \eqref{1bbl} is the function $\mathcal{F}$, which is
$$
\mathcal{F} \begin{pmatrix} v \\w \end{pmatrix} = - \mathcal{P}^{-1} 
\begin{pmatrix} (I- b \partial _x^2 )^{-1} \partial_x \left [ ( v - w )^p \mathcal{H} (v+ w) \right ]  \\ (I- b \partial _x^2 )^{-1} 
\left ( \partial _x (v- w)^{p+1}/(p+1) \right )    \end{pmatrix}\, .
$$
Here, $\mathcal{P}^{-1}, \mathcal{H}$ are bounded operators in $H^s (\mathbb{R}) \times  H^s (\mathbb{R}) $ or $H^s (\mathbb{R}) $ with $s \geq 0$. Moreover, it is straightforward to check that 
$$
\mathcal{F} \begin{pmatrix} 0 \\ 0 \end{pmatrix} =\begin{pmatrix} 0 \\0 \end{pmatrix}
$$ and
$$
\left \| \mathcal{F} \begin{pmatrix} f_1 \\ g_1\end{pmatrix} -\mathcal{F} \begin{pmatrix} f_2 \\ g_2\end{pmatrix} \right \|_{H^s(\mathbb{R})\times H^s(\mathbb{R})}\leq C R^p\| (f_1, g_1) - ( f_2, g_2) \|_{H^s(\mathbb{R})\times H^s(\mathbb{R})}
$$
where $ p \geq 1$, $s > 1/2$ and $(f_1, g_1), (f_2, g_2)$ are selected from a closed ball of radius $R$ centered at $0$ in the corresponding space. Then, the rest of the local well-posedness proof of \eqref{1bbl} follows exactly the same way as the proof of \cite[Theorem 2.5]{BCS04} using the contraction mapping theorem. Hence, we can conclude that 
for any constant $\delta_0 > 0$, there is a real constant $T_0 > 0$  depending on $\delta_0$  such that if $\| (\eta_0 , u _0)\|_{H^1(\mathbb{R})\times H^1(\mathbb{R})} \leq \delta_0$, then the solution of \eqref{1bbl} exists for $ t \in [ 0 , T_0]$ satisfying that
$$
\| (\eta (t, \cdot) , u (t, \cdot) )\|_{H^1(\mathbb{R})\times H^1(\mathbb{R})} \leq 2 \delta_0
$$
for all $t \in [ 0 , T_0]$.

For the global existence result of  \eqref{1bbl}, we can use the Hamiltonian structure of the system together with the extension of the existence interval when initial conditions are small. In particular, by the Sobolev imbedding theorem, given any solution of \eqref{1bbl} with $t \in [ 0, t_1]$, there are two fixed positive constants $c_0, c_1$ independent of the solution such that for any $t \in [ 0, t_1]$, 
$$
c_0 \| (\eta , u )\|_{H^1(\mathbb{R})\times H^1(\mathbb{R})} ^2\left ( 1 - c_1
\| (\eta , u )\|_{H^1(\mathbb{R})\times H^1(\mathbb{R})} ^p\right ) \leq \mathcal{H} \begin{pmatrix}\eta  \\ u \end{pmatrix}= \mathcal{H} \begin{pmatrix}\eta_0 \\ u_0 \end{pmatrix}\, .
$$
Therefore, if $1 - c_1
\| (\eta , u )\|_{H^1(\mathbb{R})\times H^1(\mathbb{R})} ^p \geq 1/2$ or
$\| (\eta , u )\|_{H^1(\mathbb{R})\times H^1(\mathbb{R})} \leq (1/(2c_1))^{1/p}$, 
then
$$
\| (\eta , u )\|_{H^1(\mathbb{R})\times H^1(\mathbb{R})}  \leq \left ( \frac 2{c_0}\mathcal{H} \begin{pmatrix}\eta_0 \\ u_0 \end{pmatrix} \right )^{1/2} \qquad \mbox{with  }\quad \mathcal{H} \begin{pmatrix}\eta_0 \\ u_0 \end{pmatrix} \geq 0 \, .
$$
From the local well-posedness result stated above, we select $\delta_0> 0$ with $2\delta_0  \leq (1/(2c_1))^{1/p} $ and let the initial condition satisfying 
$$
\| (\eta_0 , u _0)\|_{H^1(\mathbb{R})\times H^1(\mathbb{R})} \leq \delta_0
\quad \mbox{and}\quad \left ( \frac 2{c_0}\mathcal{H} \begin{pmatrix}\eta_0 \\ u_0 \end{pmatrix} \right )^{1/2} \leq \delta_0\, .
$$
Then, the solution of \eqref{1bbl} exists for $t\in [0, T_0]$ and satisfies
$
\| (\eta , u )\|_{H^1(\mathbb{R})\times H^1(\mathbb{R})}  \leq \delta_0\, 
$ for all $t \in [0, T_0]$. Hence, from the local well-posedness result again, we can extend the solution to $[0, 2T_0]$. Continuing this procedure yields a global solution for $t \in [0, \infty)$, which is unique and satisfies $
\| (\eta , u )\|_{H^1(\mathbb{R})\times H^1(\mathbb{R})}  \leq \delta_0\, 
$ for all $t \in [0, \infty)$.

\subsection{The proof of Lemma \ref{prepare}}
 In this subsection, we will provide a brief idea of the proof of Lemma \ref{prepare}.
 
 Let $(\psi_{\omega_0}, v_{\omega_0})\in \mathcal{G}_{\omega_0}$ and $\omega$ be close to $\omega_0$. We need to show that there is a minimizer $(\psi_{\omega}, v_{\omega}) $ near $(\psi_{\omega_0}, v_{\omega_0})$ satisfying the required conditions. First write 
 $$
 (\psi_{\omega}, v_{\omega}) = (\psi_{\omega_0}, v_{\omega_0}) + (\psi, v) \quad \mbox{and consider }\    (\psi_{\omega_0} + \psi , v_{\omega_0}    + v) + t ( \tilde \psi , \tilde v) 
 $$ 
 where $t$ is small and $( \tilde \psi , \tilde v)$ is arbitrary.  Then, 
 \begin{align*}
 G( & \psi_{\omega_0} + \psi + t  \tilde \psi ,  v_{\omega_0}    + v+ t \tilde v )  
 = \frac 2{p+1} \bigg ( \int_\mathbb{R} \psi_{\omega_0} v^{p+1}_{\omega_0} dx \\
& + \int_\mathbb{R} \psi_{\omega_0} \left ( ( v_{\omega_0}    + v+ t \tilde v )^{p+1}  - v^{p+1}_{\omega_0} \right ) dx +  \int_\mathbb{R} ( \psi + t\tilde \psi ) ( v_{\omega_0}    + v+ t \tilde v )^{p+1} dx \bigg ) \\
=& -1 + \frac 2{p+1}  \left ( \int_\mathbb{R} \psi_{\omega_0} \left ( ( v_{\omega_0}    + v+ t \tilde v )^{p+1}  - v^{p+1}_{\omega_0} \right ) dx +  \int_\mathbb{R} ( \psi + t\tilde \psi ) ( v_{\omega_0}    + v+ t \tilde v )^{p+1} dx \right ) \\
=& - ( 1 -  g) \, .
\end{align*}
Hence, to find a minimizer near $(\psi_{\omega_0}, v_{\omega_0)}$ for \eqref{mp} locally, we need to consider the function 
$$
\frac1{(1-g)^{\frac1{p+2}}}( \psi_{\omega_0} + \psi + t  \tilde \psi ,  v_{\omega_0}    + v+ t \tilde v )  \, 
$$
 and make the minimum of 
\begin{align*}
 I_\omega (t)= & \frac1{(1-g)^{\frac2{p+2}}}\bigg [ \int_\mathbb{R}\Big ( ( \psi_{\omega_0} + \psi + t  \tilde \psi )^2 - c ( \psi'_{\omega_0} + \psi '+ t  \tilde \psi ' )^2 \\
 & \qquad\qquad\qquad + ( v_{\omega_0}    + v+ t \tilde v )^2 
  - a (v'_{\omega_0}    + v'+ t \tilde v' )^2  \Big )dx \\
  & - 2 \omega \int_\mathbb{R}\Big ( ( \psi_{\omega_0} + \psi + t  \tilde \psi ) ( v_{\omega_0}    + v+ t \tilde v ) + b ( \psi'_{\omega_0} + \psi '+ t  \tilde \psi ' ) (v'_{\omega_0}    + v'+ t \tilde v' ) \Big ) dx \bigg ]\, .
\end{align*}
 Thus,  we can take the $t$-derivative of $I_\omega (t)$ and let $t=0$ to get variational equations for $(\psi , v)$. Since $(\tilde \psi, \tilde v)$ is arbitrary in $X$, we can obtain a system of $(\psi_{\omega_0} + \psi  ,  v_{\omega_0}    + v)$, which is an unscaled version of \eqref{trav-ep}, i.e.,
 \begin{equation} 
 \left \{ 
 \begin{array}{rl}
  (\psi_{\omega_0} + \psi ) + c (\psi_{\omega_0} + \psi )'' - \omega ( v_{\omega_0} + v ) + \omega b ( v_{\omega_0} + v ) ''   \   \\ \\
 +  \frac2{(p+2)(p+1) ( 1- g )|_{t =0} } ( v_{\omega_0} + v ) ^{p+1}  = 0 \, , \\ \\
  ( v_{\omega_0} + v ) +  ( v_{\omega_0} + v )'' - \omega (\psi_{\omega_0} + \psi ) + a (\psi_{\omega_0} + \psi )''    \   \\ \\
   + \frac2{(p+2) ( 1- g)|_{t=0} }    (\psi_{\omega_0} + \psi )  ( v_{\omega_0} + v ) ^{p}  =0 \, . \end{array} 
  \right . \label{ayyyyy}
\end{equation}
 It is known that when $\omega -\omega_0=0$, \eqref{ayyyyy} has a solution $(\psi , v) = (0, 0)$. Then, we use a similar idea in Appendix \ref{A} and linearize \eqref{ayyyyy} around $(\psi_{\omega_0},  v_{\omega_0})$ to obtain a unique solution $(\psi , v) $ when $\omega -\omega_0 $ is small. The detailed argument can be found in Appendix \ref{A}.  Since the solution of \eqref{ayyyyy} near zero is unique and $I_\omega$ has a minimum when $\omega $ is near $\omega_0$, such a solution of \eqref{ayyyyy} must be a minimizer of \eqref{mp}, at least locally. Moreover, since the system \eqref{ayyyyy} depends on $\omega$ analytically, the solution $(\psi, v)  $ is continuously differentiable with respect to $\omega$ for small enough $\omega-\omega_0$.

\subsection*{Acknowledgment} The authors are grateful to the referees for the careful reading of this paper and their valuable suggestions and comments.
This work was done while the first author was visiting Virginia Tech. The author thanks the host institution for their warm hospitality.

\end{document}